\documentclass[a4paper,12pt]{article}
\usepackage[T1]{fontenc}
\usepackage{amsmath}
\usepackage{amssymb}
\usepackage{amsthm}
\usepackage{tikz}
\usetikzlibrary{patterns}
\usepackage{setspace}

\newtheorem{theorem}{Theorem}[section]
\newtheorem{definitio}[theorem]{Definition}
\newenvironment{definition}{\begin{definitio} \rm }{\end{definitio}}
\newtheorem{lemma}[theorem]{Lemma}
\newtheorem{proposition}[theorem]{Proposition}
\newtheorem{corollary}[theorem]{Corollary}

\newtheorem{question}{Question}

\newenvironment{remark}{\begin{proof}[Remark] }{\end{proof}}

\newcommand{\fl}[1]{\textrm{\raisebox{-0.12cm}{\begin{tikzpicture}\draw (0.5,0) node{$\longrightarrow$}; \draw (0.5,0.17) node{$\scriptstyle{#1}$};
  \end{tikzpicture}}}}

\newcommand{\N}{\mathbb{N}}
\newcommand{\Z}{\mathbb{Z}}
\newcommand{\Q}{\mathbb{Q}}
\newcommand{\K}{\mathbb{K}}
\newcommand{\Qt}{\Q[t^{\pm1}]}
\newcommand{\Zt}{\mathbb{Z}[t^{\pm1}]}
\newcommand{\F}{\mathcal{F}}
\newcommand{\G}{\mathcal{G}}

\newcommand{\A}{\mathcal{A}}
\newcommand{\Ens}{\mathcal{P}}
\newcommand{\Al}{\mathfrak{A}}
\newcommand{\bl}{\mathfrak{b}}

\newcommand{\AutZ}{\textrm{Aut}_\Z}
\newcommand{\expd}{\textrm{exp}_\sqcup}
\newcommand{\Abb}{\overline{\Al},\overline{\bl}}

\newcommand{\psib}{\overline{\psi}}
\newcommand{\QS}{\textrm{$\scriptstyle{\Q}$s}}

\title{Finite type invariants of knots in homology 3--spheres with respect to null LP--surgeries}

\author{Delphine Moussard}

\date{}

\begin{document}

\maketitle

\begin{abstract}
We study a theory of finite type invariants for null-homologous knots in rational homology 3--spheres with respect to null Lagrangian-preserving surgeries. 
It is an analogue in the setting of the rational homology of the Garoufalidis--Rozansky theory for knots in integral homology 3--spheres. 
We give a partial combinatorial description of the graded space associated with our theory and determine some cases when this description is complete. 
For null-homologous knots in rational homology 3--spheres with a trivial Alexander polynomial, we show that the Kricker lift of the Kontsevich integral 
and the Lescop equivariant invariant built from integrals in configuration spaces are universal finite type invariants for this theory; 
in particular it implies that they are equivalent for such knots. 
\vspace{1ex}

\vspace{1ex}

\end{abstract}

\tableofcontents

    \section{Introduction}

The notion of finite type invariants was first introduced independently by Goussarov and Vassiliev for the study of invariants of knots in the 3--dimensional 
sphere $S^3$; in this case, finite type invariants are also called Vassiliev invariants. 
The discovery of the Kontsevich integral, which is a universal invariant among all finite type invariants of knots in $S^3$, revealed that this class 
of invariants is very prolific. It is known, for instance, that it dominates all Witten--Reshetikhin--Turaev's quantum invariants. 
The notion of finite type invariants was adapted to the setting of 3--dimensional manifolds by Ohtsuki \cite{Oht4}, who introduced the first examples for integral 
homology 3--spheres, and it has been widely developed and generalized since then. In particular, Goussarov and Habiro independently developed 
a theory which involves any 3--dimensional manifolds ---and their knots--- and which contains the Ohtsuki theory for $\Z$--spheres (see Garoufalidis--Goussarov--Polyak \cite{GGP} and Habiro \cite{Hab}). 
Another generalization of the Ohtsuki theory to general 3--dimensional manifolds was developed by Cochran and Melvin \cite{CM}. 

In general, the finite type invariants of a set of objects are defined by their polynomial behavior with respect to some elementary move. 
For Vassiliev invariants of knots in $S^3$, this move is the crossing change on a diagram of the knot. For 3--dimensional manifolds, the elementary 
move is a certain kind of surgery, for instance the Borromean surgery ---a Lagrangian-preserving replacement of a genus $3$ handlebody--- in the Goussarov--Habiro theory. 

In \cite{GR}, Garoufalidis and Rozansky studied the theory of finite type invariants for $\Z$SK--pairs, {\em i.e.} knots 
in integral homology 3--spheres, with respect to the so-called null-move, which is a Borromean surgery defined on a handlebody that is null-homologous 
in the complement of the knot. In this paper, we study a theory of finite type invariants for $\Q$SK--pairs, {\em i.e.} null-homologous knots in rational 
homology 3--spheres ($\Q$--spheres). Our elementary move is the null Lagrangian-preserving surgery introduced by Lescop \cite{Les3}, which is the Lagrangian-preserving 
replacement of a rational homology handlebody that is null-homologous in the complement of the knot. This latter theory can be understood as an adaptation 
of the Garoufalidis--Rozansky theory to the setting of the rational homology; a great part of the results in this paper are stated in both settings.

In \cite{Kri}, Kricker constructed a rational lift of the Kontsevich integral of $\Z$SK--pairs. 
In \cite{GK}, he proved with Garoufalidis that his construction provides an invariant of $\Z$SK--pairs. This invariant takes values in a diagram space 
with a stronger structure than the target diagram space of the Kontsevich integral, hence it is much more structured than the Kontsevich integral, which it lifts. 
Garoufalidis and Kricker proved in \cite{GK} that the Kricker invariant satisfies some splitting formulas with respect to the null-move (see also 
\cite{GR}). 
These formulas imply in particular that the Kricker invariant is a series of finite type invariants of all degrees with respect to the null-move. 

It appears that the null-move preserves the Blanchfield module ---the Alexander module equipped with the Blanchfield form--- of 
the $\Z$SK--pair. Hence the study of the Garoufalidis--Rozansky theory of finite type invariants can be restricted to a class of $\Z$SK--pairs with a fixed Blanchfield 
module. In the case of a trivial Blanchfield module, Garoufalidis and Rozansky gave a combinatorial description of the associated graded space. 
Together with the splitting formulas of Garoufalidis and Kricker, this proves that the Kricker invariant is a universal finite type invariant of $\Z$SK--pairs 
with trivial Blanchfield module with respect to the null-move. 

Another universal invariant in this context was constructed by Lescop in \cite{Les2}. Lescop proved in \cite{Les3} that her invariant satisfies the same splitting 
formulas as the Kricker invariant. Hence the Lescop invariant is also a universal finite type invariant of $\Z$SK--pairs with trivial 
Blanchfield module with respect to the null-move. It implies in particular that the Lescop invariant and the Kricker invariant are equivalent for 
$\Z$SK--pairs with trivial Blanchfield module. Lescop conjectured in \cite{Les3} that this equivalence holds for knots with any Blanchfield module. 

The Lescop invariant is indeed defined for $\Q$SK--pairs and Lescop's splitting formulas are stated with respect to general null 
Lagrangian-preserving surgeries. In \cite{M6}, the Kricker invariant is extended to $\Q$SK--pairs and splitting formulas for this 
invariant with respect to null Lagrangian-preserving surgeries are given. Hence a combinatorial description of the graded space associated 
with finite type invariants of $\Q$SK--pairs with respect to null Lagrangian-preserving surgeries would allow to explicit the universality 
properties of these two invariants and provide a comparison between them, answering the above conjecture of Lescop for general $\Q$SK--pairs. 

In analogy with the integral homology setting, null Lagrangian-preserving surgeries preserve the Blanchfield module defined over $\Q$ and we study finite 
type invariants of $\Q$SK--pairs with a fixed Blanchfield module. In the case of a trivial Blanchfield module, we give a complete description of 
the associated graded space. This description and the above-mentioned splitting formulas imply that the Lescop invariant and the Kricker invariant are both universal 
finite type invariants of $\Q$SK--pairs with trivial Blanchfield module, up to degree one invariants given by the cardinality of the first homology group 
of the $\Q$--sphere. In particular, the Lescop invariant and the Kricker invariant are equivalent for $\Q$SK--pairs with trivial Blanchfield module 
when the cardinality of the first homology group of the $\Q$--sphere is fixed. 

Let $(\Al,\bl)$ be any Blanchfield module with annihilator $\delta\in\Qt$. The main goal of this paper is to give a combinatorial description of the graded space 
$\G(\Al,\bl)=\oplus_{n\in\Z}\G_n(\Al,\bl)$ associated with finite type invariants of $\Q$SK--pairs with Blanchfield module $(\Al,\bl)$ ---precise definitions are given 
in the next section. The Lescop or Kricker invariant $Z=(Z_n)_{n\in\N}$ is a family of finite type invariants $Z_n$ of degree $n$ for $n$ even ($Z_n$ is trivial 
for $n$ odd). For $\Q$SK--pairs with Blanchfield module $(\Al,\bl)$, $Z_n$ takes values in a space $\A_n(\delta)$ of trivalent graphs with edges 
labelled in $\frac{1}{\delta}\Qt$. The finiteness properties imply that $Z_n$ induces a map on $\G_n(\Al,\bl)$. In order to take into account 
the degree 1 invariants, we construct from $Z$ an invariant $Z^{aug}=(Z_n^{aug})_{n\in\N}$ of $\Q$SK--pairs with $Z_n^{aug}$ of degree $n$. The invariant 
$Z_n^{aug}$ takes values in a space $\A_n^{aug}(\delta)$ of trivalent graphs as before, which may in addition contain isolated vertices labelled by prime integers. 
Again by finiteness, $Z_n^{aug}$ induces a map on $\G_n(\Al,\bl)$. This leads us to our main question.
\begin{question} \label{quconj}
 Is the map $Z_n^{aug}:\G_n(\Al,\bl)\to\A_n^{aug}(\delta)$ injective ?
\end{question}
Injectivity of this map for any Blanchfield module $(\Al,\bl)$ is equivalent to universality of the invariant $Z$ as a finite type invariant of $\Q$SK--pairs, up to 
degree 0 and 1 invariants. This would imply the equivalence of the Lescop invariant and the Kricker invariant when the Blanchfield module 
and the cardinality of the first homology group of the $\Q$--sphere are fixed. 

To deal with Question \ref{quconj}, we first construct another diagram space $\A_n^{aug}(\Al,\bl)$ together with a surjective map 
$\varphi_n:\A_n^{aug}(\Al,\bl)\twoheadrightarrow\G_n(\Al,\bl)$. 
Then we compose this map with $Z_n^{aug}$ to get a map $\psi_n:\A_n^{aug}(\Al,\bl)\to\A_n^{aug}(\delta)$. 
\begin{figure} [htb]
 \begin{center}
 \begin{tikzpicture} [xscale=1.4,yscale=0.8]
  \draw (2,2) node {$\G_n(\Al,\bl)$};
  \draw (2,-2) node {$\A_n^{aug}(\delta)$};
  \draw (0,0) node {$\A_n^{aug}(\Al,\bl)$};
  \draw[->>] (0.4,0.5) -- (1.6,1.5); \draw (0.9,1.3) node {$\varphi_n$};
  \draw[->] (0.4,-0.5) -- (1.6,-1.5); \draw (0.9,-1.3) node {$\psi_n$};
  \draw[->] (2,1.5) -- (2,-1.5); \draw (2.4,0) node {$Z_n^{aug}$};
 \end{tikzpicture}
 \end{center}
\caption{Commutative diagram}
\end{figure}
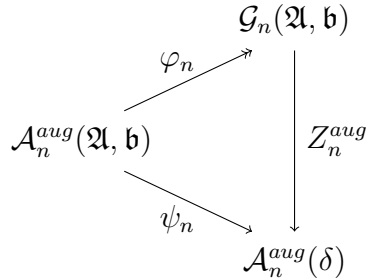
It appears that this composed map has a simple diagrammatic description. Nevertheless, it is not easy to decide whether it is injective or not in general. 
\begin{question} \label{qupsin}
 Is the map $\psi_n:\A_n^{aug}(\Al,\bl)\to\A_n^{aug}(\delta)$ injective ?
\end{question}
If Question \ref{qupsin} has a positive answer, then Question \ref{quconj} also has, and $\G_n(\Al,\bl)$ is completely described combinatorially by 
$\varphi_n:\A_n^{aug}(\Al,\bl)\fl{\cong}\G_n(\Al,\bl)$. 

Question \ref{qupsin} has a positive answer at least in the following cases ---the last two cases are treated with Audoux in \cite{AM}:
\begin{itemize}
 \item for a trivial Blanchfield module and any value of $n$,
 \item for a Blanchfield module which is a direct sum of $N$ isomorphic Blanchfield modules and $n\leq\frac{2}{3}N$,
 \item for a Blanchfield module of $\Q$--dimension 2 and $n=2$,
 \item for a Blanchfield module which is a direct sum of two isomorphic Blanchfield modules of $\Q$--dimension 2 and of order different from $t+1+t^{-1}$, and $n=2$.
\end{itemize}
In the third case, the map $\psi_n$ is not surjective, whereas in the other cases, it is an isomorphism. In particular, $Z_n^{aug}$ is not surjective in general. 
Moreover, for a Blanchfield module which is a direct sum of two isomorphic Blanchfield modules of $\Q$--dimension 2 and of order $t+1+t^{-1}$, and $n=2$, 
Question \ref{qupsin} has a negative answer \cite{AM}, but Question \ref{quconj} is open ---as well as the injectivity status of $\varphi$.

The fact that Question \ref{quconj} remains open while Question \ref{qupsin} does not have a positive answer in general leads us to the following alternative:
\begin{itemize}
 \item either Question \ref{quconj} has a positive answer in general, in which case $\G_n(\Al,\bl)$ is isomorphic to $\psi_n(\A_n^{aug}(\Al,\bl))$,
 \item or we miss some invariant to add to the augmented Lescop/Kricker invariant and the Blanchfield module to get a universal finite type invariant of $\Q$SK--pairs.
\end{itemize}

We also treat the Garoufalidis--Rozansky theory of finite type invariants of $\Z$SK--pairs in the case of a non-trivial Blanchfield module.

\paragraph{Notations}
For $\K=\Z,\Q$, a {\em $\K$--sphere} (resp. {\em $\K$--ball}, {\em $\K$--torus}, {\em genus $g$ $\K$--handlebody}) is a compact connected oriented 
3--manifold with the same homology with coefficients in $\K$ as the standard 3--sphere (resp. 3--ball, solid torus, genus $g$ standard handlebody). 
A $\K$SK--pair $(M,K)$ is a pair made of a $\K$--sphere $M$ and a knot $K$ in $M$ whose homology class in $H_1(M;\Z)$ is trivial.

\paragraph{Plan of the paper}
In Section \ref{secresults}, we introduce the necessary notions and state the main results of the paper. Section \ref{secborro} is devoted to clasper 
calculus in the equivariant setting. We apply this calculus in Section \ref{secdiagrams} to our diagrams. 
This provides a surjective map from a graded diagram space to the graded space associated with $\Z$SK--pairs 
with respect to integral null Lagrangian-preserving surgeries. To get a similar map in the case of $\Q$SK--pairs, we need 
further arguments developed in Section \ref{secsurjective}. In Section \ref{secZLes}, we show the universality property of the invariant $Z^{aug}$ 
which combines the Lescop/Kricker invariant and the cardinality of the first homology group. In Section~\ref{secpsi}, we answer Question \ref{qupsin} 
for a Blanchfield module which is a direct sum of $N$ isomorphic Blanchfield modules in degree at most $\frac{2}{3}N$.

\paragraph{Acknowledgments.}
I began this work during my PhD thesis at the Institut Fourier of the University of Grenoble. I warmly thank my advisor, Christine Lescop, for her help and advice 
during my PhD and since then. I am currently supported by a Postdoctoral Fellowship of the Ja\-pan Society for the Promotion of Science. 
I am grateful to Tomotada Ohtsuki and the Research Institute for Mathematical Sciences for their support. 
While working on the contents of this article, I have also been supported by the Italian FIRB project ``Geometry and topology of low-dimensional
manifolds'', RBFR10GHHH.

    \section{Statement of the results} \label{secresults}

  \subsection{Filtration defined by null LP--surgeries} \label{subsecsurgeries}

We first recall the definition of the Alexander module and the Blanchfield form. 
Let $(M,K)$ be a $\Q$SK--pair. Let $T(K)$ be a tubular neighborhood of $K$. The \emph{exterior} of $K$ is 
$X=M\setminus Int(T(K))$. Consider the projection $\pi : \pi_1(X) \to \frac{H_1(X;\mathbb{Z})}{torsion} \cong \mathbb{Z}$ 
and the covering map $p : \tilde{X} \to X$ associated with its kernel. The covering $\tilde{X}$ is the \emph{infinite cyclic covering} 
of $X$. The automorphism group of the covering, $Aut(\tilde{X})$, is isomorphic to $\mathbb{Z}$. It acts on 
$H_1(\tilde{X};\Q)$. Denoting the action of a generator $\tau$ of $Aut(\tilde{X})$ as the multiplication by $t$, 
we get a structure of $\Qt$--module on $\Al(M,K)=H_1(\tilde{X};\Q)$. This $\Qt$--module is called the \emph{Alexander module} of $(M,K)$. 
It is a torsion $\Qt$--module. 

On the Alexander module, the \emph{Blanchfield form}, or \emph{equivariant linking pairing}, 
$\bl : \Al\times\Al \to \frac{\Q(t)}{\Qt}$, is defined as follows. First define the equivariant linking number of two knots.
Let $J_1$ and $J_2$ be two knots in $\tilde{X}$ such that $p(J_1)\cap p(J_2)=\emptyset$. Let $\delta\in\Q(t)$ be the annihilator of $\Al$.
There is a rational 2--chain $S$ such that $\partial S=\delta(\tau)J_1$. The \emph{equivariant linking number} of $J_1$ and $J_2$ is 
$$lk_e(J_1,J_2)=\frac{1}{\delta(t)}\sum_{k\in\Z}\langle S,\tau^k(J_2)\rangle t^k,$$
where $\langle .,.\rangle$ stands for the algebraic intersection number. 
It is well-defined and we have $lk_e(J_1,J_2)\in\frac{1}{\delta(t)}\Qt$, $lk_e(J_2,J_1)(t)=lk_e(J_1,J_2)(t^{-1})$ and 
$lk_e(P(\tau)J_1,Q(\tau)J_2)(t)=P(t)Q(t^{-1})lk_e(J_1,J_2)(t)$.
Now, if $\gamma$ (resp. $\eta$) is the homology class of $J_1$ (resp. $J_2$) in $\Al$, define $\bl(\gamma,\eta)$ by:
$$\bl(\gamma,\eta)=lk_e(J_1,J_2)\ mod\ \Qt.$$
The Blanchfield form is {\em hermitian}: $\bl(\gamma,\eta)(t)=\bl(\eta,\gamma)(t^{-1})$ and 
$\bl(P(t)\gamma,Q(t)\eta)(t)=P(t)Q(t^{-1})\,\bl(\gamma,\eta)(t)$ for all $\gamma,\eta\in\Al$ and all $P,Q\in\Qt$. 
Moreover, it is {\em non degenerate} (see Blanchfield in \cite{Bla}) : $\bl(\gamma,\eta)=0$ for all $\eta\in\Al$ implies $\gamma=0$. 

The Alexander module of a $\Q$SK--pair $(M,K)$ endowed with its Blanchfield form is its {\em Blanchfield module} denoted by $(\Al,\bl)(M,K)$. 
In the sequel, by {\em a Blanchfield module $(\Al,\bl)$}, we mean a pair $(\Al,\bl)$ which can be realized as the Blanchfield module of a $\Q$SK--pair. 
An isomorphism between Blanchfield modules is an isomorphism between the underlying Alexander modules which preserves the Blanchfield form.

We now define LP--surgeries. 
Note that the boundary of a genus $g$ $\Q$--handlebody is homeomorphic to the standard genus $g$ surface. 
The \emph{Lagrangian} $\mathcal{L}_A$ of a $\Q$--handlebody $A$ is the kernel of the map $i_*: H_1(\partial A;\Q)\to H_1(A;\Q)$ 
induced by the inclusion. Two $\Q$--handlebodies $A$ and $B$ have \emph{LP--identified} boundaries if $(A,B)$ is equipped with a homeomorphism 
$h:\partial A\to\partial B$ such that $h_*(\mathcal{L}_A)=\mathcal{L}_B$.
The Lagrangian of a $\Q$--handlebody $A$ is indeed a Lagrangian subspace of $H_1(\partial A;\Q)$ 
with respect to the intersection form.

Let $M$ be a $\Q$--sphere, let $A\subset M$ be a $\Q$--handlebody and let $B$ be a $\Q$--handlebody whose boundary is LP--identified 
with $\partial A$. Set $M\left(\frac{B}{A}\right)=(M\setminus Int(A))\cup_{\partial A=_h\partial B}B$. We say that the $\Q$--sphere 
$M\left(\frac{B}{A}\right)$ is obtained from $M$ by \emph{Lagrangian-preserving surgery}, or \emph{LP--surgery}.
 
Given a $\Q$SK--pair $(M,K)$, a \emph{$\Q$--handlebody null in $M\setminus K$} is a $\Q$--handlebody $A\subset M\setminus K$ such that 
the map $i_* : H_1(A;\Q)\to H_1(M\setminus K;\Q)$ induced by the inclusion has a trivial image.
A \emph{null LP--surgery} on $(M,K)$ is an LP--surgery $\left(\frac{B}{A}\right)$ such that $A$ is null in $M\setminus K$. 
The $\Q$SK--pair obtained by surgery is denoted by $(M,K)\left(\frac{B}{A}\right)$. 

Let $\F_0$ be the rational vector space generated by all $\Q$SK--pairs up to orientation-preser\-ving homeomorphism. 
Let $\F_n$ be the subspace of $\F_0$ generated by the 
$$\left[(M,K);\left(\frac{B_i}{A_i}\right)_{1\leq i \leq n}\right]=\sum_{I\subset \{ 1,...,n\}} (-1)^{|I|} (M,K)\left(\left(\frac{B_i}{A_i}\right)_{i\in I}\right)$$ 
for all $\Q$SK--pairs $(M,K)$ and all families of $\Q$--handlebodies $(A_i,B_i)_{1\leq i \leq n}$, where the $A_i$ are null in $M\setminus K$ 
and disjoint, and each $\partial B_i$ is LP--identified with the corresponding $\partial A_i$. Here and in all the article, 
$|.|$ stands for the cardinality. Since $\F_{n+1}\subset \F_n$, this defines a filtration. 

\begin{definition} \label{definv}
 A $\Q$--linear map $\lambda: \F_0 \to \Q$ is a \emph{finite type invariant of degree at most $n$ of $\Q$SK--pairs with respect to null LP--surgeries} 
 if $\lambda(\F_{n+1})=0$. 
\end{definition}

\begin{theorem}[\cite{M3} Theorem 1.14] \label{thM3}
 A null LP--surgery induces a canonical isomorphism between the Blanchfield modules of the involved $\Q$SK--pairs. Conversely, for any isomorphism $\zeta$ 
 from the Blanchfield module of a $\Q$SK--pair $(M,K)$ to the Blanchfield module of a $\Q$SK--pair $(M',K')$, there is a finite sequence of null LP--surgeries 
 from $(M,K)$ to $(M',K')$ which induces the composition of $\zeta$ by the multiplication by a power of $t$.
\end{theorem}
This result provides a splitting of the filtration $(\F_n)_{n\in\N}$, as follows. 
For an isomorphism class $(\Al,\bl)$ of Blanchfield modules, let $\Ens(\Al,\bl)$ be the set of all $\Q$SK--pairs, up to orientation-preserving homeomorphism, 
whose Blanchfield modules are isomorphic to $(\Al,\bl)$. Let $\F_0(\Al,\bl)$ be the subspace of $\F_0$ generated by the $\Q$SK--pairs $(M,K)\in\Ens(\Al,\bl)$. 
Let $(\F_n(\Al,\bl))_{n\in\N}$ be the filtration defined on $\F_0(\Al,\bl)$ by null LP--surgeries. Then, for $n\in\N$, $\F_n$ is 
the direct sum over all isomorphism classes $(\Al,\bl)$ of Blanchfield modules of the $\F_n(\Al,\bl)$. Set $\G_n(\Al,\bl)=\F_n(\Al,\bl) / \F_{n+1}(\Al,\bl)$ 
and $\G(\Al,\bl)=\oplus_{n\in\N}\G_n(\Al,\bl)$. We wish to describe the graded space $\G(\Al,\bl)$. By Theorem \ref{thM3}, $\G_0(\Al,\bl)\cong\Q$. 
In Section \ref{secsurjective}, as a consequence of Theorem \ref{thextend}, we prove:
\begin{theorem} \label{thG1}
 Let $(\Al,\bl)$ be a Blanchfield module. Let $(M,K)\in\Ens(\Al,\bl)$. 
For any prime integer $p$, let $B_p$ be a $\Q$--ball such that $H_1(B_p;\Z)\cong\Z/p\Z$. Then:
$$\G_1(\Al,\bl)=\bigoplus_{p\textrm{ prime}}\Q \left[(M,K);\frac{B_p}{B^3}\right],$$
where $B^3$ is any standard 3--ball in $M\setminus K$.
\end{theorem}

   \subsection{Borromean surgeries}

Let us define a specific type of LP--surgeries.

The {\em standard Y--graph} is the graph $\Gamma_0\subset \mathbb{R}^2$ represented in Figure \ref{figY4}. 
The looped edges of $\Gamma_0$ are the \emph{leaves}. 
The vertex incident to three different edges is the {\em internal vertex}. 
With $\Gamma_0$ is associated a regular neighborhood $\Sigma(\Gamma_0)$ of $\Gamma_0$ in the plane. 
The surface $\Sigma(\Gamma_0)$ is oriented with the usual convention. This induces an orientation of the leaves 
and an orientation of the internal vertex, {\em i.e.} a cyclic order of the three edges. 
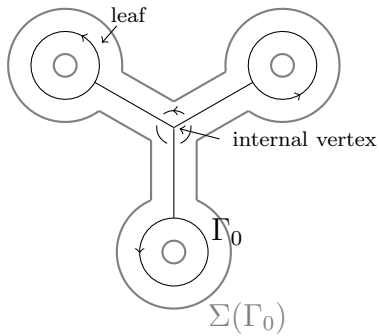
\begin{figure}[htb] 
\begin{center}
\begin{tikzpicture} [scale=0.15]
\newcommand{\feuille}[1]{
\draw[rotate=#1,thick,color=gray] (0,-11) circle (5);
\draw[rotate=#1,thick,color=gray] (0,-11) circle (1);
\draw[rotate=#1,line width=8pt,color=white] (-2,-6.42) -- (2,-6.42);
\draw[rotate=#1,thick,color=gray] (2,-1.15) -- (2,-6.42);
\draw[rotate=#1,thick,color=gray] (-2,-1.15) -- (-2,-6.42);
\draw[white,rotate=#1,line width=5pt] (0,0) -- (0,-8);
\draw[rotate=#1] (0,0) -- (0,-8);
\draw[rotate=#1] (0,-11) circle (3);
\draw[->,rotate=#1] (-3,-10.9) -- (-3,-11.1);}
\draw (0,0) circle (1.5);
\draw[->] (0.1,1.5) -- (-0.1,1.5);
\feuille{0}
\feuille{120}
\feuille{-120}
\draw (-4,10) node{$\scriptstyle{\textrm{leaf}}$};
\draw[->] (-5,9) -- (-6.3,7.5);
\draw (11.5,-1) node{$\scriptstyle{\textrm{internal vertex}}$};
\draw[<-] (0.5,-0.1) -- (4,-1);
\draw (4.7,-9.2) node{$\Gamma_0$};
\draw[color=gray] (6.5,-16.8) node{$\Sigma(\Gamma_0)$};
\end{tikzpicture}
\end{center}
\caption{The standard Y--graph}\label{figY4}
\end{figure}
Consider a 3--manifold $M$ and an embedding $h:\Sigma(\Gamma_0)\to M$. The image $\Gamma$ of $\Gamma_0$ is a {\em Y--graph}, endowed with  
its {\em associated surface} $\Sigma(\Gamma)=h(\Sigma(\Gamma_0))$. The Y--graph $\Gamma$ is equipped with 
the framing induced by $\Sigma(\Gamma)$. 
A \emph{Y--link} in a 3--manifold is a collection of disjoint Y--graphs. 

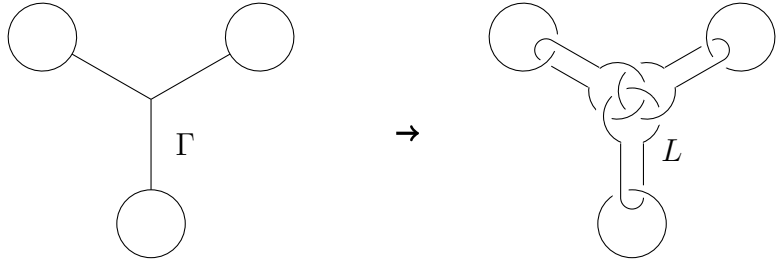
\begin{figure}[htb] 
\begin{center}
\begin{tikzpicture} [scale=0.15]
\begin{scope}
\newcommand{\feuille}[1]{
\draw[rotate=#1] (0,0) -- (0,-8);
\draw[rotate=#1] (0,-11) circle (3);}
\feuille{0}
\feuille{120}
\feuille{-120}
\draw (3,-4) node{$\Gamma$};
\end{scope}
\draw[very thick,->] (21.5,-3) -- (23.5,-3);
\begin{scope}[xshift=1200]
\newcommand{\bras}[1]{
\draw[rotate=#1] (0,-1.5) circle (2.5);
\draw [rotate=#1,white,line width=8pt] (-0.95,-4) -- (0.95,-4);
\draw[rotate=#1] {(0,-11) circle (3) (1,-3.9) -- (1,-7.6)};
\draw[rotate=#1,white,line width=6pt] (-1,-5) -- (-1,-8.7);
\draw[rotate=#1] {(-1,-3.9) -- (-1,-8.7) (-1,-8.7) arc (-180:0:1)};}
\bras{0}
\draw [white,line width=6pt,rotate=120] (0,-1.5) circle (2.5);
\bras{120}
\draw [rotate=-120,white,line width=6pt] (-1.77,0.27) arc (135:190:2.5);
\draw [rotate=-120,white,line width=6pt] (1.77,0.27) arc (45:90:2.5);
\bras{-120}
\draw [white,line width=6pt] (-1.77,0.27) arc (135:190:2.5);
\draw [white,line width=6pt] (1.77,0.27) arc (45:90:2.5);
\draw (-1.77,0.27) arc (135:190:2.5);
\draw (1.77,0.27) arc (45:90:2.5);
\draw (3.5,-4.5) node{$L$};
\end{scope}
\end{tikzpicture}
\end{center}
\caption{Y--graph and associated surgery link}\label{figborro4}
\end{figure}

Let $\Gamma$ be a Y--graph in a 3--manifold $M$. Let $\Sigma(\Gamma)$ be its associated surface. In $\Sigma\times[-1,1]$, 
associate with $\Gamma$ the six components link $L$ represented in Figure \ref{figborro4}. 
The \emph{Borromean surgery on $\Gamma$} is the surgery along the framed link $L$. The surgered manifold is denoted $M(\Gamma)$. 
As proved by Matveev in \cite{Mat}, a Borromean surgery can be realised by cutting a genus 3 handlebody 
(a regular neighborhood of the Y--graph) and regluing it in another way, which preserves the Lagrangian. 
If $(M,K)$ is a $\Q$SK--pair and if $\Gamma$ is an $n$--component Y--link, null in $M\setminus K$, 
then $\left[(M,K);\Gamma\right]\in\F_0$ denotes the bracket defined by the $n$ disjoint null LP--surgeries on the components of $\Gamma$. 

For $n\geq 0$, let $\G_n^b(\Al,\bl)$ be the subspace of $\G_n(\Al,\bl)$ generated by the classes of the brackets defined by null Borromean surgeries. 
The following result is a consequence of Proposition \ref{propphinintro} and Lemma \ref{lemmaodd}.
\begin{proposition} \label{propquotientsimpairs}
 For any Blanchfield module $(\Al,\bl)$ and any $n\geq0$, $\G_{2n+1}^b(\Al,\bl)=0$. 
\end{proposition}

  \subsection{Spaces of diagrams} \label{subsecdiag}

Fix a Blanchfield module $(\Al,\bl)$. Let $\delta\in\Q(t)$ be the annihilator of $\Al$. 
An {\em $(\Al,\bl)$--colored diagram} $D$ is a uni-trivalent graph without strut (\raisebox{-1.7ex}{\begin{tikzpicture} 
\draw (0,-0.5) -- (0,0);
\draw (0,-0.5) node{$\scriptstyle{\bullet}$};
\draw (0,0) node{$\scriptstyle{\bullet}$};
\end{tikzpicture}}), with the following data:
\begin{itemize}
 \item trivalent vertices are oriented, where an {\em orientation of a trivalent vertex} is a cyclic order 
  of the three half-edges that meet at this vertex; 
 \item edges are oriented and colored by $\Qt$;
 \item univalent vertices are colored by $\Al$;
 \item for all $v\neq v'$ in the set $V$ of univalent vertices of $D$, a rational fraction $f_{vv'}^D(t)\in\frac{1}{\delta(t)}{\Qt}$ is fixed 
  such that $f_{vv'}^D(t)\ mod\ \Qt=\bl(\gamma,\gamma')$, where $\gamma$ (resp. $\gamma'$) is the coloring of $v$ (resp. $v'$), with 
  $f_{v'v}^D(t)=f_{vv'}^D(t^{-1})$.
\end{itemize}
In the pictures, the orientation of trivalent vertices is given by 
  \raisebox{-1.5ex}{\begin{tikzpicture} [scale=0.2]
  \newcommand{\tiers}[1]{
  \draw[rotate=#1,color=white,line width=4pt] (0,0) -- (0,-2);
  \draw[rotate=#1] (0,0) -- (0,-2);}
  \draw (0,0) circle (1);
  \draw[<-] (-0.05,1) -- (0.05,1);
  \tiers{0}
  \tiers{120}
  \tiers{-120}
  \end{tikzpicture}}. 
When it does not seem to cause confusion, we write $f_{vv'}$ for $f_{vv'}^D$. 
The {\em degree} of a colored diagram is the number of trivalent vertices of its underlying graph. The unique degree 0 diagram is the empty diagram. 
For $n\geq0$, set: 
$$\tilde{\A}_n(\Al,\bl)=\frac{\Q \langle(\Al,\bl)\mbox{--colored diagrams of degree $n$}\rangle}{\Q \langle\mbox{AS, IHX, LE, OR, Hol, LV, EV, LD}\rangle},$$ 
where the relations AS (anti-symmetry), IHX, LE (linearity for edges), OR (orientation reversal), Hol (holonomy), 
LV (linearity for vertices), EV (edge-vertex) and LD (linking difference) are described in Figure \ref{figrel}. 
\begin{figure}[htb] 
\begin{center}
\begin{tikzpicture} [scale=0.3]
\draw (0,4) -- (2,2);
\draw (2,2) -- (4,4);
\draw (2,2) -- (2,0);
\draw (5,2) node{+};
\draw (8,2) .. controls +(2,0) and +(2.5,-1) .. (6,4);
\draw[white,line width=6pt] (8,2) .. controls +(-2,0) and +(-2.5,-1) .. (10,4);
\draw (8,2) .. controls +(-2,0) and +(-2.5,-1) .. (10,4);
\draw (8,0) -- (8,2);
\draw (11,2) node{=};
\draw (12.5,2.1) node{0};
\draw (6,-1.5) node{AS};
\draw (18,4) -- (20,3) -- (20,1) -- (18,0);
\draw (20,1) -- (22,0);
\draw (20,3) -- (22,4);
\draw (23,2) node{-};
\draw (24,4) -- (25,2) -- (27,2) -- (28,4);
\draw (24,0) -- (25,2);
\draw (27,2) -- (28,0);
\draw (29,2) node{+};
\draw (30,4) -- (33,2) -- (34,0);
\draw[white,line width=6pt] (31,2) -- (34,4);
\draw (30,0) -- (31,2) -- (34,4);
\draw (31,2) -- (33,2);
\draw (35,2) node{=};
\draw (36.5,2.1) node{0};
\draw (20.5,2) node{1};
\draw (26,1.2) node{1};
\draw (32,1.2) node{1};
\draw (27,-1.5) node{IHX};
\begin{scope} [xshift=-4cm,yshift=-8cm]
\draw (0,2) node{$x$};
\draw (1,0) -- (1,4);
\draw[->] (1,0) -- (1,3);
\draw (1.8,2.8) node{$P$};
\draw (3.2,2) node{$+$};
\draw (4.7,2) node{$y$};
\draw (5.7,0) -- (5.7,4);
\draw[->] (5.7,0) -- (5.7,3);
\draw (6.5,2.8) node{$Q$};
\draw (7.9,2) node{$=$};
\draw (9.4,0) -- (9.4,4);
\draw[->] (9.4,0) -- (9.4,3);
\draw (12.4,2.8) node{$xP+yQ$};
\draw (6.2,-1.5) node{LE};
\end{scope}
\begin{scope} [xshift=15cm,yshift=-8.5cm]
\draw (0,0) -- (0,4);
\draw[->] (0,0) -- (0,3);
\draw (1.5,2.8) node{$P(t)$};
\draw (3.5,2) node{$=$};
\draw (5,0) -- (5,4);
\draw[->] (5,4) -- (5,3);
\draw (7.2,2.8) node{$P(t^{-1})$};
\draw (3.5,-1.5) node{OR};
\end{scope}
\begin{scope} [xshift=30cm,yshift=-7cm]
\newcommand{\edge}[2]{
\draw[rotate=#1] (0,0) -- (0,3);
\draw[rotate=#1,->] (0,0) -- (0,1.5);
\draw[rotate=#1] (1,1.3) node{#2};}
\edge{0}{$P$}
\edge{120}{$Q$}
\edge{240}{$R$}
\draw (3.8,0) node{$=$};
\begin{scope} [xshift=7.6cm]
\edge{0}{$tP$}
\edge{120}{$tQ$}
\edge{240}{$tR$}
\end{scope}
\draw (3.8,-3.5) node{Hol};
\end{scope}
\begin{scope} [xshift=-4cm,yshift=-17cm]
\draw (-1,2) node{$x$};
\draw (1,0) -- (1,4);
\draw (1,1) node[right] {$D_1$};
\draw (1,4) node{$\bullet$};
\draw (1,4) node[right] {$\gamma_1$};
\draw (1,4) node[below left] {$\scriptstyle{v}$};
\draw (4,2) node{$+$};
\draw (5.5,2) node{$y$};
\draw (7.5,0) -- (7.5,4);
\draw (7.5,1) node[right] {$D_2$};
\draw (7.5,4) node{$\bullet$};
\draw (7.5,4) node[right] {$\gamma_2$};
\draw (7.5,4) node[below left] {$\scriptstyle{v}$};
\draw (10.5,2) node{$=$};
\draw (13,0) -- (13,4);
\draw (13,1) node[right] {$D$};
\draw (13,4) node{$\bullet$};
\draw (13,4) node[right] {$x\gamma_1+y\gamma_2$};
\draw (13,4) node[below left] {$\scriptstyle{v}$};
\draw (8.5,-2) node{$xf^{D_1}_{vv'}(t)+yf^{D_2}_{vv'}(t)=f^{D}_{vv'}(t)\quad \forall\, v'\neq v$};
\draw (7,-4.5) node{LV};
\end{scope}
\begin{scope} [xshift=22cm,yshift=-17cm]
\draw (-0.8,0) -- (-0.8,4);
\draw (-0.8,4) node{$\bullet$};
\draw (-0.8,4) node[below left] {$\scriptstyle{v}$};
\draw[->] (-0.8,0) -- (-0.8,2);
\draw (-0.8,4) node[right] {$\gamma$};
\draw (0.6,1.8) node{$PQ$};
\draw (-0.8,0) node[right] {$D$};
\draw (3,2) node{$=$};
\draw (5,0) -- (5,4);
\draw (5,4) node{$\bullet$};
\draw (5,4) node[below left] {$\scriptstyle{v}$};
\draw[->] (5,0) -- (5,2);
\draw (5,4) node[right] {$Q(t)\gamma$};
\draw (6,1.8) node{$P$};
\draw (5,0) node[right] {$D'$};
\draw (10,3) node[right] {$f_{vv'}^{D'}(t)=Q(t)f_{vv'}^D(t)$};
\draw (12,1) node[right] {$\forall\, v'\neq v$};
\draw (7,-2) node{EV};
\end{scope}
\begin{scope} [xshift=10cm,yshift=-29cm]
\draw (-1,0) -- (-1,4);
\draw (-1,2) node[right] {$1$};
\draw (-1,4) node{$\bullet$};
\draw (-1,4) node[below left] {$\scriptstyle{v_1}$};
\draw (-1,4) node[right] {$\gamma_1$};
\draw (2.5,0) -- (2.5,4);
\draw (2.5,2) node[right] {$1$};
\draw (2.5,4) node{$\bullet$};
\draw (2.5,4) node[below left] {$\scriptstyle{v_2}$};
\draw (2.5,4) node[right] {$\gamma_2$};
\draw (2.5,0) node[right] {$D$};
\draw (5.5,2) node{$=$};
\draw (8,0) -- (8,4);
\draw (8,2) node[right] {$1$};
\draw (8,4) node{$\bullet$};
\draw (8,4) node[below left] {$\scriptstyle{v_1}$};
\draw (8,4) node[right] {$\gamma_1$};
\draw (11.5,0) -- (11.5,4);
\draw (11.5,2) node[right] {$1$};
\draw (11.5,4) node{$\bullet$};
\draw (11.5,4) node[below left] {$\scriptstyle{v_2}$};
\draw (11.5,4) node[right] {$\gamma_2$};
\draw (11.5,0) node[right] {$D'$};
\draw (14.5,2) node{$+$};
\draw (16,0) -- (16,2.75) (18.5,0) -- (18.5,2.75) arc (0:180:1.25);
\draw[->] (17.2,4) -- (17.3,4);
\draw (17.25,4) node[above] {$P$};
\draw (18.5,0) node[right] {$D''$};
\draw (7.5,-2) node {$f^D_{v_1v_2}=f^{D'}_{v_1v_2}+P$};
\draw (8,-4.5) node {LD};
\end{scope}
\end{tikzpicture}
\end{center}
\caption{Relations, where $x,y\in\Q$, $P,Q,R\in\Qt$, $\gamma,\gamma_1,\gamma_2\in\Al$.} \label{figrel}
\end{figure}

The automorphism group $Aut(\Al,\bl)$ of the Blanchfield module $(\Al,\bl)$ acts on $\tilde{\A}_n(\Al,\bl)$ 
by acting on the colorings of all the univalent vertices of a diagram 
simultaneously. Denote by Aut the relation which identifies two diagrams obtained from one another by the action of an element 
of $Aut(\Al,\bl)$. Set:
$$\A_n(\Al,\bl)=\tilde{\A}_n(\Al,\bl)/\langle \textrm{Aut} \rangle \qquad \textrm{and} \qquad \A(\Al,\bl)=\oplus_{n\in\N}\A_n(\Al,\bl).$$ 

Since the opposite of the identity is an automorphism of $(\Al,\bl)$, we have:
\begin{lemma} \label{lemmaodd}
 For all $n\geq 0$, $\A_{2n+1}(\Al,\bl)=0$. 
\end{lemma}

In Section \ref{secdiagrams}, we prove (Proposition \ref{propphin}):
\begin{proposition} \label{propphinintro}
 Fix a Blanchfield module $(\Al,\bl)$. 
 For all $n\geq 0$, there is a canonical surjective $\Q$--linear map: $$\varphi_n : \A_n(\Al,\bl) \twoheadrightarrow \G_n^b(\Al,\bl).$$
\end{proposition}
To get a similar surjective map onto $\G_n(\Al,\bl)$, we need more general diagrams. 
An {\em $(\Al,\bl)$--augmented diagram} is the union of an $(\Al,\bl)$--colored diagram (its {\em Jacobi part}) and of finitely many isolated 
vertices colored by prime integers. The {\em degree} of an $(\Al,\bl)$--augmented diagram is the number of its vertices of valence $0$ or $3$. 
Set:
$$\A_n^{aug}(\Al,\bl)=\frac{\Q \langle\mbox{$(\Al,\bl)$--augmented diagrams of degree $n$}\rangle}{\Q \langle\mbox{AS, IHX, LE, OR, Hol, LV, EV, LD, Aut}\rangle}\quad\textrm{for }n\geq 0,$$
$$\A^{aug}(\Al,\bl)=\oplus_{n\in\N}\A_n^{aug}(\Al,\bl).$$

In Section \ref{secsurjective}, we prove:
\begin{theorem} \label{thextend}
 Fix a Blanchfield module $(\Al,\bl)$. 
 For all $n\geq 0$, there is a canonical surjective $\Q$--linear map: $$\varphi_n : \A_n^{aug}(\Al,\bl) \twoheadrightarrow \G_n(\Al,\bl).$$
\end{theorem}

We will see in the next subsection that this map is an isomorphism when the Blanchfield module $(\Al,\bl)$ is trivial.

    \subsection{The Lescop invariant and the Kricker invariant} \label{subsecLesKri}

In order to introduce the Kricker invariant of \cite{GK} and the Lescop invariant of \cite{Les2}, we first define the graded space $\A(\delta)$ 
where they takes values and we relate it to the graded space $\A(\Al,\bl)$. 

Let $\delta\in\Qt$. A {\em $\delta$--colored diagram} is a trivalent graph whose vertices are oriented and whose edges 
are oriented and colored by $\frac{1}{\delta(t)}\Qt$. The degree of a $\delta$--colored diagram is the number of its vertices. 
Set:
$$\A_n(\delta)=\frac{\Q\langle\mbox{$\delta$--colored diagrams of degree $n$}\rangle}{\Q\langle\mbox{AS, IHX, LE, OR, Hol, Hol'}\rangle},$$
where the relations AS, IHX, LE, OR, Hol are represented in Figure \ref{figrel} and the relation Hol' is represented in Figure \ref{fighol'}. Here the relations 
LE, OR and Hol hold with edges labelled in $\frac{1}{\delta(t)}\Qt$. Note that in the case of $\A_n(\Al,\bl)$, the relation Hol' is induced by the 
relations Hol, EV and LD. 
\begin{figure}[htb] 
\begin{center}
\begin{tikzpicture} [scale=0.3]
\draw (0,0) -- (0,3);
\draw (0,-1) circle (1);
\draw[->] (-0.1,-2) -- (0,-2);
\draw (0,-3) node{$g$};
\draw[->] (0,0) -- (0,1.5);
\draw (0,1.3) node[right] {$f$};
\draw (3.8,0) node{$=$};
\begin{scope} [xshift=7.6cm]
\draw (0,0) -- (0,3);
\draw (0,-1) circle (1);
\draw[->] (-0.1,-2) -- (0,-2);
\draw (0,-3) node{$g$};
\draw[->] (0,0) -- (0,1.5);
\draw (0,1.3) node[right] {$tf$};
\end{scope}
\end{tikzpicture}
\end{center}
\caption{Relation Hol', with $f,g\in\frac{1}{\delta(t)}\Qt$.} \label{fighol'}
\end{figure}
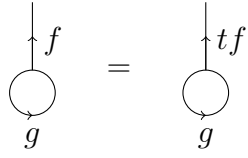
Since any trivalent graph has an even number of vertices, we have $\A_{2n+1}(\delta)=0$ for all $n\geq0$.

With an $(\Al,\bl)$--colored diagram $D$ of degree $n$, we associate a $\delta$--colored diagram $\tilde{\psi}_n(D)$. 
Let $V$ be the set of univalent vertices of $D$. A {\em pairing of $V$} is an involution of $V$ 
with no fixed point. Let $\mathfrak{p}$ be the set of pairings of $V$. Fix $p\in\mathfrak{p}$. 
Define a $\delta$--colored diagram $p(D)$ in the following way. If $v\in V$ and $v'=p(v)$, replace in $D$ 
the vertices $v$ and $v'$, and their adjacent edges, by a colored edge, as indicated in Figure \ref{figgroup}. 
Now set:
$$\tilde{\psi}_n(D)=\sum_{p\in\mathfrak{p}} p(D).$$
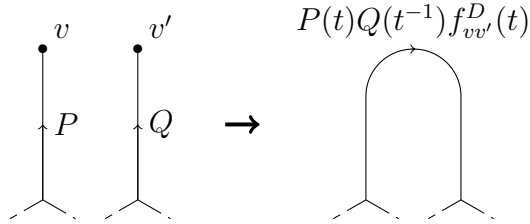
\begin{figure}[htb] 
\begin{center}
\begin{tikzpicture} [scale=0.5]
\draw (0,0) -- (0,4) (2.5,0) -- (2.5,4);
\draw (0,4) node {$\scriptstyle{\bullet}$} (2.5,4) node {$\scriptstyle{\bullet}$};
\draw (0,4) node[above right] {$v$} (2.5,4) node[above right] {$v'$};
\draw[->] (0,0) -- (0,2);
\draw[->] (2.5,0) -- (2.5,2);
\draw (0,2) node[right] {$P$} (2.5,2) node[right] {$Q$};
\draw (-0.43,-0.25) -- (0,0) -- (0.43,-0.25) (2.07,-0.25) -- (2.5,0) -- (2.93,-0.25);
\draw[dashed] (-0.86,-0.5) -- (-0.43,-0.25) (0.43,-0.25) -- (0.86,-0.5) (1.64,-0.5) -- (2.07,-0.25) (2.93,-0.25) -- (3.36,-0.5);
\draw[->,line width=1.5pt] (4.8,2) -- (5.7,2);
\begin{scope} [xshift=8.5cm]
\draw (-0.43,-0.25) -- (0,0) -- (0.43,-0.25) (2.07,-0.25) -- (2.5,0) -- (2.93,-0.25);
\draw[dashed] (-0.86,-0.5) -- (-0.43,-0.25) (0.43,-0.25) -- (0.86,-0.5) (1.64,-0.5) -- (2.07,-0.25) (2.93,-0.25) -- (3.36,-0.5);
\draw (0,0) -- (0,2.75) (2.5,0) -- (2.5,2.75) arc (0:180:1.25);
\draw[->] (1.2,4) -- (1.3,4);
\draw (1.25,4) node[above] {$P(t)Q(t^{-1})f_{vv'}^D(t)$};
\end{scope}
\end{tikzpicture}
\end{center}
\caption{Pairing of vertices} \label{figgroup}
\end{figure}
Note that $\tilde{\psi}_n(D)=0$ when the number of univalent vertices is odd. 
We obtain a $\Q$--linear map $\tilde{\psi}_n$ from the rational vector space freely generated by the 
$(\Al,\bl)$--colored diagrams of degree $n$ to $\A_n(\delta)$. One easily checks that $\tilde{\psi}_n$ induces a map:
$$\psi_n:\A_n(\Al,\bl)\to\A_n(\delta).$$

The disjoint union of diagrams defines a multiplicative operation on $\A(\delta)=\oplus_{n\in\N}\A_n(\delta)$ which endows it with a graded algebra structure.
Denote by $\expd$ the exponential map with respect to this multiplication. 

The following result asserts the existence and the properties of an invariant $Z$ which may be either the Lescop invariant or the Kricker invariant. Althought it is not known 
whether they are equal or not, they both satisfy the properties of the theorem. In the sequel, we will refer to ``the invariant $Z$''.
\begin{theorem}[\cite{Les2,Les3,Kri,GK,M6}] \label{thinvariantZ}
 There is an invariant $Z=(Z_n)_{n\in\N}$ of $\Q$SK--pairs with the following properties. 
 \begin{itemize}
  \item If $(M,K)$ is a $\Q$SK--pair with Blanchfield module $(\Al,\bl)$, then $Z_n(M,K)\in\A_n(\delta)$, where $\delta$ is the annihilator of $\Al$. 
  \item Fix a Blanchfield module $(\Al,\bl)$. Let $\delta$ be the annihilator of $\Al$. The $\Q$--linear extension of $Z_n:\Ens(\Al,\bl)\to\A_n(\delta)$ to $\F_0(\Al,\bl)$ 
  vanishes on $\F_{n+1}(\Al,\bl)$ and $Z_n\circ\varphi_n=\psi_n$, where $\varphi_n:\A_n(\Al,\bl)\twoheadrightarrow\G_n^b(\Al,\bl)$ is the surjection of Proposition \ref{propphinintro}.
  \item Let $p_n^c:\A_n(\delta)\to\A_n(\delta)$ be the map which sends a connected diagram to itself and non-connected diagrams to 0. Set $Z_n^c=p^c\circ Z_n$ and 
   $Z^c=\sum_{n>0}Z_n^c$. Then $Z^c$ is additive under connected sum and $Z=\expd(Z^c)$. 
 \end{itemize}
\end{theorem}
We will detail the second statement of this theorem in Section \ref{secdiagrams}.
Note that, in particular, if the map $\psi_n$ is injective, then the map $\varphi_n$ is an isomorphism. 

In order to take into account the whole quotient $\G_n(\Al,\bl)$, we extend the invariant~$Z$. 
Define a {\em $\delta$--augmented diagram} as the disjoint union of a $\delta$--colored diagram with finitely many isolated vertices 
colored by prime integers. The {\em degree} of such a diagram is the number of its vertices. Set:
$$\A_n^{aug}(\delta)=\frac{\Q\langle\mbox{$\delta$--augmented diagrams of degree $n$}\rangle}{\Q\langle\mbox{AS, IHX, LE, OR, Hol, Hol'}\rangle}.$$
The map $\psi_n$ naturally extends to a map $\psi_n:\A_n^{aug}(\Al,\bl)\to\A_n^{aug}(\delta)$ preserving the isolated vertices. 
We now define an invariant $Z^{aug}=(Z_n^{aug})_{n\in\N}$ of $\Q$SK--pairs such that the $\Q$--linear extension of $Z_n^{aug}$ to $\F_0(\Al,\bl)$ takes 
values in $\A_n^{aug}(\delta)$, from which the invariant $Z$ is recovered by forgetting the isolated vertices. For a prime integer $p$, 
define an invariant $\rho_p$ by $\rho_p(M,K)=-v_p(|H_1(M;\Z)|).\bullet_p$, where $v_p$ is the $p$--adic valuation. 
Once again, the disjoint union makes $\A^{aug}(\delta)=\oplus_{n\in\N}\A_n^{aug}(\delta)$ 
a graded algebra. Set:
$$Z^{aug}=Z\sqcup\expd\left(\sum_{p\textrm{ prime}}\rho_p\right).$$

In Section \ref{secZLes}, we prove:
\begin{theorem} \label{thunivinv}
 Fix a Blanchfield module $(\Al,\bl)$. Let $\delta$ be the annihilator of $\Al$. Consider the surjection $\varphi_n:\A_n^{aug}(\Al,\bl)\to\G_n(\Al,\bl)$ 
 of Theorem \ref{thextend} and the map $\psi_n:\A_n^{aug}(\Al,\bl)\to\A_n^{aug}(\delta)$. Then the $\Q$--linear extension of $Z_n^{aug}:\Ens(\Al,\bl)\to\A_n^{aug}(\delta)$ 
 to $\F_0(\Al,\bl)$ vanishes on $\F_{n+1}(\Al,\bl)$ and $Z_n^{aug}\circ\varphi_n=\psi_n$.
\end{theorem}

Let $\Al_0$ be the trivial Blanchfield module. The relations LV and LD allow to express the elements of $\A_n^{aug}(\Al_0)$ without diagrams 
with univalent vertices. It follows that this diagram space as a simpler presentation as:
$$\A_n^{aug}(\Al_0)=\frac{\Q\langle\mbox{augmented diagrams of degree $n$}\rangle}{\Q\langle\mbox{AS, IHX, LE, OR, Hol, Hol'}\rangle},$$
where an {\em augmented diagram} is the disjoint union of a trivalent part ---a trivalent graph whose vertices are oriented and whose edges 
are oriented and colored by $\Qt$--- and a finite number of isolated vertices colored by prime integers. The degree of an augmented diagram is the number of its vertices. 
The space $\A_n(\Al_0)$ admits the similar description without isolated vertices; the corresponding graded space $\A(\Al_0)$ coincides 
with the space denoted $\A(\Qt)$ in \cite{GR}.
Obviously, $\psi_n:\A_n^{aug}(\Al_0)\to\A_n^{aug}(1)$ is an isomorphism. 
Hence Theorems \ref{thextend} and \ref{thunivinv} imply the next results.
\begin{theorem} \label{th0}
 We have a graded space isomorphism $\G(\Al_0)\cong\A^{aug}(\Al_0)$.
\end{theorem}

\begin{theorem} \label{thequiv}
 Let $Z_{Les}=(Z_{n,Les})_{n\in\N}$ and $Z_{Kri}=(Z_{n,Kri})_{n\in\N}$ denote the Lescop equivariant invariant and the Kricker invariant respectively. 
 Let $(M,K)$ and $(N,J)$ be $\Q$SK--pairs with trivial Blanchfield module such that $H_1(M;\Z)$ and $H_1(N;\Z)$ have the same cardinality. 
 Then, for any $n\in\N$, $Z_{k,Les}(M,K)=Z_{k,Les}(N,J)$ for all $k\leq n$ if and only if $Z_{k,Kri}(M,K)=Z_{k,Kri}(N,J)$ for all $k\leq n$.
\end{theorem}
\begin{proof}
 Let $Z=(Z_n)_{n\in\N}$ be the Lescop or Kricker invariant. Since $H_1(M;\Z)$ and $H_1(N;\Z)$ have the same cardinality, the assertion 
 ``$Z_k(M,K)=Z_k(N,J)$ for all $k\leq n$'' is equivalent to ``$Z_k^{aug}(M,K)=Z_k^{aug}(N,J)$ for all $k\leq n$''. Since the 
 $Z_k^{aug}:\G_k(\Al_0)\to\A_k^{aug}(\Al_0)$ are isomorphisms, this last assertion is equivalent to ``$(M,K)-(N,J)\in\F_{n+1}(\Al_0)$''.
\end{proof}

In general, note that ``the map $\psi_n:\A_n^{aug}(\Al,\bl)\to\A_n^{aug}(\delta)$ is injective'' is equivalent to ``the map $\psi_k:\A_k(\Al,\bl)\to\A_k(\delta)$ is 
injective for all $k\leq n$''. Hence we focus on the study of injectivity of the map $\psi_n:\A_n(\Al,\bl)\to\A_n(\delta)$.

      \subsection{About the map $\psi_n:\A_n(\Al,\bl)\to\A_n(\delta)$ and perspectives}

We now state a result about the injectivity of the map $\psi_n:\A_n(\Al,\bl)\to\A_n(\delta)$ for $n$ even. 

In Section \ref{secpsi}, we prove:
\begin{theorem} \label{th3n}
 Let $n$ be an even positive integer and $N\geq3n/2$. Fix a Blanchfield module $(\Al,\bl)$. Let $\delta$ be the annihilator of $\Al$. Define the Blanchfield 
 module $(\Abb)$ as the direct sum of $N$ copies of $(\Al,\bl)$. Then the map $\psib_n:\A_n(\Abb)\to\A_n(\delta)$ is an isomorphism.
\end{theorem}

This result provides a rewritting of the map $\psi_n$ in the general case. We have a natural map $\iota_n:\A_n(\Al,\bl)\to\A_n(\Abb)$ 
defined on a diagram by interpreting the labels of its univalent vertices as elements of the first copy of $(\Al,\bl)$ in $(\Abb)$. 
The following diagram commutes.
$$
\begin{tikzpicture} [xscale=1.4,yscale=0.8]
 \draw (2,2) node {$\A_n(\Abb)$};
 \draw (2,-2) node {$\A_n(\delta)$};
 \draw (0,0) node {$\A_n(\Al,\bl)$};
 \draw[->>] (0.4,0.5) -- (1.6,1.5); \draw (0.9,1.3) node {$\iota_n$};
 \draw[->] (0.4,-0.5) -- (1.6,-1.5); \draw (0.9,-1.3) node {$\psi_n$};
 \draw[->] (2,1.5) -- (2,-1.5); \draw (2.3,0) node {$\psib_n$}; \draw (1.8,0) node {$\cong$};
\end{tikzpicture}$$

We mention here results from \cite{AM} about the map $\psi_2:\A_2(\Al,\bl)\to\A_2(\delta)$ for small Alexander modules.
\begin{proposition}[\cite{AM}]
 If $\dim_\Q(\Al)=2$, then $\psi_2$ is injective but not surjective.
\end{proposition}

\begin{proposition}[\cite{AM}]
 If $\Al$ is the direct sum of two isomorphic Blanchfield modules of $\Q$--dimension 2 with annihilator $\delta$, 
 then $\psi_2$ is injective if and only if $\delta\neq t+1+t^{-1}$. In this case, it is an isomorphism. 
\end{proposition}

\paragraph{Perspectives.}
As mentioned in the introduction, our main goal in this paper is to study Question \ref{quconj} in order to determine if the Lescop/Kricker invariant $Z$ is a 
universal finite type invariant of $\Q$SK--pairs up to degree~0 and~1 invariants. Theorem \ref{th3n} provides the following rewritting of this question.

We have a map $\A_n((\Al,\bl)^{\oplus k})\to\A_n((\Al,\bl)^{\oplus k+1})$ defined by viewing the labels of the univalent vertices 
in the direct sum of the first $k$ copies of $(\Al,\bl)$ in $(\Al,\bl)^{\oplus k+1}$. We also have a map $C_n:\G_n^b((\Al,\bl)^{\oplus k})\to\G_n^b((\Al,\bl)^{\oplus k+1})$ 
induced by the connected sum with a fixed $\Q$SK--pair $(M,K)\in\Ens(\Al,\bl)$. Using Theorem \ref{thM3}, one can check that the map $C_n$ is independent of the fixed pair $(M,K)$. 
These maps provide the following commutative diagram, where the vertical arrows are the maps $\varphi_n$ and $Z_n$, for any integer $N$ 
such that $N\geq\frac{3n}{2}$. 
\begin{center}
\begin{tikzpicture} 
 \draw (0,0) node {$\A_n(\Al,\bl)$};
 \draw[->] (1,0) -- (1.4,0);
 \draw (1.9,0) node {$\dots$};
 \draw[->] (2.3,0) -- (2.7,0);
 \draw (4.25,0) node {$\A_n((\Al,\bl)^{\oplus k})$};
 \draw[->] (5.75,0) -- (6.15,0);
 \draw (6.55,0) node {$\dots$};
 \draw[->] (6.9,0) -- (7.3,0);
 \draw (9,0) node {$\A_n((\Al,\bl)^{\oplus N})$};
 \draw[->>] (0,-0.5) -- (0,-1);
 \draw[->>] (4.25,-0.5) -- (4.25,-1);
 \draw[->>] (9,-0.5) -- (9,-1);
 \draw (9.4,-0.8) node {$\scriptstyle{\cong}$};
 \draw (0,-1.5) node {$\G_n^b(\Al,\bl)$};
 \draw[->] (1,-1.5) -- (1.4,-1.5);
 \draw (1.9,-1.5) node {$\dots$};
 \draw[->] (2.3,-1.5) -- (2.7,-1.5);
 \draw (4.25,-1.5) node {$\G_n^b((\Al,\bl)^{\oplus k})$};
 \draw[->] (5.75,-1.5) -- (6.15,-1.5);
 \draw (6.55,-1.5) node {$\dots$};
 \draw[->] (6.9,-1.5) -- (7.3,-1.5);
 \draw (9,-1.5) node {$\G_n^b((\Al,\bl)^{\oplus N})$};
 \draw[->] (0.5,-2) -- (3.5,-2.8);
 \draw[->] (4.25,-2) -- (4.25,-2.7);
 \draw (4.25,-3) node {$\A_n(\delta)$};
 \draw[->] (8.5,-2) -- (5.1,-2.8);
 \draw (7,-2.6) node {$\scriptstyle{\cong}$};
\end{tikzpicture}
\end{center}
It follows that the map $Z_n:\G_n^b((\Al,\bl)^{\oplus k})\to\A_n(\delta)$ is injective for all $k$ if and only if 
$C_n:\G_n^b((\Al,\bl)^{\oplus k})\to\G_n^b((\Al,\bl)^{\oplus k+1})$ is injective for all $k$. 
This assertion is true for all $(\Al,\bl)$ and all $n$ if the space of finite type invariants of $\Q$SK--pairs is generated 
as an algebra by degree 0 invariants and invariants that are additive under connected sum.

    \subsection{The case of knots in $\Z$--spheres}

A great part of the results stated up to that point have an equivalent in the case of $\Z$SK--pairs. In this subsection, we adapt the definitions and state the results 
in this case. 

Given a $\Z$SK--pair $(M,K)$ and the infinite cyclic covering $\tilde{X}$ of the exterior of $K$ in $M$, define the {\em integral Alexander module} of $(M,K)$ 
as the $\Z[t^{\pm1}]$--module $\Al_\Z(M,K)=H_1(\tilde{X};\Z)$ and the {\em Blanchfield form} $\bl_\Z$ on this module.
The integral Alexander module of a $\Z$SK--pair $(M,K)$ endowed with its Blanchfield form is its {\em integral Blanchfield module} denoted by $(\Al_\Z,\bl_\Z)(M,K)$. 
In the sequel, by {\em an integral Blanchfield module}, we mean a pair $(\Al_\Z,\bl_\Z)$ which can be realized as the integral 
Blanchfield module of a $\Z$SK--pair. 

Replacing $\Q$ by $\Z$ in the definitions of Subsection \ref{subsecsurgeries}, define {\em integral Lagrangians}, 
{\em integral LP--surgeries} and {\em integral null LP--surgeries}. Note that a Borromean surgery is an integral LP--surgery. 

For diagram spaces, we have to adapt the relation Aut. 
Given $(\Al_\Z,\bl_\Z)$, set $(\Al,\bl)=\Q\otimes(\Al_\Z,\bl_\Z)$. Define the relation $\AutZ$ on $(\Al,\bl)$--colored diagrams as the relation Aut restricted 
to the action of the automorphisms in $Aut(\Al,\bl)$ that are induced by automorphisms of the $\Zt$--module $(\Al_\Z,\bl_\Z)$. Set:
$$\A_n^\Z(\Al_\Z,\bl_\Z)=\tilde{\A}_n(\Al,\bl)/\langle \AutZ \rangle \qquad \textrm{and} \qquad 
\A^\Z(\Al_\Z,\bl_\Z)=\oplus_{n\in\N}\A_n^\Z(\Al_\Z,\bl_\Z).$$ 
Since the opposite of the identity is an automorphism of $(\Al_\Z,\bl_\Z)$, we have: 
\begin{lemma}
 For all $n\geq 0$, $\A_{2n+1}^\Z(\Al_\Z,\bl_\Z)=0$. 
\end{lemma}

The filtration $(\F_n)_{n\in\N}$ of Subsection \ref{subsecsurgeries} generalizes the following filtration introduced by Garoufalidis and Rozansky in \cite{GR}. 
Let $\F_0^\Z$ be the rational vector space generated by all $\Z$SK--pairs, up to orientation-preserving homeomorphism. Define a filtration $(\F_n^\Z)_{n\in\N}$ of $\F_0^\Z$ 
by means of null Borromean surgeries. 
\paragraph{Remark:} Habegger \cite[Theorem 2.5]{Habe} and Auclair and Lescop \cite[Lemma 4.11]{AL} proved that two $\Z$--handlebodies whose boundaries are LP--identified 
can be obtained from one another by a finite sequence of Borromean surgeries. Therefore, the filtration defined on $\F_0^\Z$ by integral null LP--surgeries 
is equal to the filtration $(\F_n^\Z)_{n\in\N}$.

\begin{theorem}[\cite{M3} Theorem 1.15] \label{thM3Z}
 An integral null LP--surgery induces a canonical isomorphism between the integral Blanchfield modules of the involved $\Z$SK--pairs. 
 Conversely, for any isomorphism $\zeta$ from the integral Blanchfield module of a $\Z$SK--pair $(M,K)$ to the integral Blanchfield module 
 of a $\Z$SK--pair $(M',K')$, there is a finite sequence of integral null LP--surgeries from $(M,K)$ to $(M',K')$ which induces the composition 
 of $\zeta$ by the multiplication by a power of $t$.
\end{theorem}
This result provides a splitting of the filtration $(\F_n^\Z)_{n\in\N}$ as the direct sum of filtrations 
$(\F_n^\Z(\Al_\Z,\bl_\Z))_{n\in\N}$ of subspaces $\F_0^\Z(\Al_\Z,\bl_\Z)$ of $\F_0^\Z$, where $(\Al_\Z,\bl_\Z)$ runs along all isomorphism 
classes of integral Blanchfield modules. Set $\G_n^\Z(\Al_\Z,\bl_\Z)=\F_n^\Z(\Al_\Z,\bl_\Z)/\F_{n+1}^\Z(\Al_\Z,\bl_\Z)$ and 
$\G^\Z(\Al_\Z,\bl_\Z)=\oplus_{n\in\N}\G_n^\Z(\Al_\Z,\bl_\Z)$. Theorem \ref{thM3Z} implies $\G_0^\Z(\Al_\Z,\bl_\Z)=\Q$.
In \cite{GR}, Garoufalidis and Rozansky identified the graded space $\G^\Z(\Al_0)$, where $\Al_0$ is the trivial Blanchfield module, 
with the graded space $\A^\Z(\Al_0)$. Theorem \ref{th0} generalizes this result.

Proposition \ref{propphinZ} implies:
\begin{theorem} \label{thphiZ}
 Fix an integral Blanchfield module $(\Al_\Z,\bl_\Z)$. 
 For all $n\geq 0$, there is a canonical surjective $\Q$--linear map: 
 $$\varphi_n^\Z : \A_n^\Z(\Al_\Z,\bl_\Z) \twoheadrightarrow \G_n^\Z(\Al_\Z,\bl_\Z).$$
\end{theorem}

\begin{corollary}
 Fix an integral Blanchfield module $(\Al_\Z,\bl_\Z)$ and an integer $n\geq 0$. Then $\G_{2n+1}^\Z(\Al_\Z,\bl_\Z)=0.$
\end{corollary}

As in Subsection \ref{subsecLesKri}, we have a map $\psi_n^\Z:\A_n^\Z(\Al_\Z,\bl_\Z)\to\A_n(\delta)$, where $\delta$ is the annihilator of $\Al=\Q\otimes\Al_\Z$. 
Theorem \ref{thinvariantZ} implies that the degree $n$ part of the invariant $Z$ provides a $\Q$--linear map $Z_n:\F_0^\Z(\Al_\Z,\bl_\Z)\to\A_n(\delta)$ such that 
$Z_n\circ\varphi_n^\Z=\psi_n^\Z$. 

Set $\bl=id_\Q\otimes\bl_\Z$. We have a natural projection $\A_n^\Z(\Al_\Z,\bl_\Z)\to\A_n(\Al,\bl)$. The map $\psi_n^\Z$ is the composition of the map $\psi_n$ with this projection. 
Hence we could adapt Theorem~\ref{th3n} and get a surjective map $\overline{\psi_n^\Z}$, but we would not get injectivity, which is what we are mostly interested in.

	\section{Equivariant clasper calculus} \label{secborro}

For a $\Q$SK--pair $(M,K)$, let $\F_0^b(M,K)$ be the rational vector space generated by all the $\Q$SK--pairs that can be obtained from 
$(M,K)$ by a finite sequence of null Borromean surgeries, up to orientation-preserving homeomorphism. For $n>0$, let $\F_n^b(M,K)$ 
be the subspace of $\F_0^b(M,K)$ generated by the $\left[(M,K);\Gamma\right]$ for all $m$--component null Y--links with $m\geq n$.
\begin{lemma}[\cite{GGP} Lemma 2.2] \label{lemmadisk4}
 Let $\Gamma$ be a Y--graph in a 3--manifold $V$, which has a 0--framed leaf that bounds a disk in $V$ whose interior does not meet $\Gamma$. 
Then $V(\Gamma)\cong V$.
\end{lemma}

\begin{lemma}[\cite{GGP} Theorem 3.1, \cite{Auc} Lemma 5.1.1] \label{lemmapreslide}
 Let $\Gamma_0$, $\Gamma_1$, $\Gamma_2$ be the Y--graphs drawn in a genus 4 handlebody in Figure \ref{figtopeq1}. Assume this handlebody 
is embedded in a 3--manifold $V$. Then $V(\Gamma_0)\cong V(\Gamma_1\cup\Gamma_2)$.
\end{lemma}
\begin{figure}[htb] 
\begin{center}
\begin{tikzpicture} [scale=0.4]
\newcommand{\handle}{\foreach \x in {0,1,2,3} 
{\draw[very thick,rotate=90*\x] (4,0) .. controls +(0,4) and +(4,0) .. (0,4) (2,2) circle (0.5);}}
\begin{scope}
\handle
\foreach \x in {0,1,2} {\draw[rotate=90*\x] (2,2) circle (1);}
\draw (-1.3,1.3) -- (0,0.5) -- (1.3,1.3) ;
\draw (1,-2) arc (-180:90:1);
\draw (0,0.5) .. controls +(0,-1) and +(-1,0) .. (2,-1) (1,-2) .. controls +(0,0.8) and +(1,1) .. (-1.3,-1.3);
\draw (2,0) node{$\Gamma_0$};
\end{scope}
\begin{scope} [xshift=12cm]
\handle
\foreach \x in {0,1,2,3} {\draw[rotate=90*\x] (2,2) circle (1);}
\draw (-1,-2) .. controls +(1,0) and +(-0.2,-0.8) .. (0.5,-0.5);
\draw (-1.95,-0.75) .. controls +(1,-0.3) and +(-1,-0.4) .. (0.5,-0.5);
\draw[white,line width=3pt] (-2,-0.5) arc (90:360:1.5);
\draw[white,line width=3pt] (-0.5,-2) arc (0:60:1.5);
\draw (-2,-2) circle (1.5);
\draw[white,line width=3pt] (-2.15,-0.2) arc (90:270:0.3);
\draw (-1.9,-0.35) arc (30:330:0.3);
\draw (-1.3,1.3) -- (0,0.5) -- (1.3,1.3) (0,0.5) -- (-1.4,-0.6);
\draw (1.3,-1.3) -- (0.5,-0.5);
\draw (0,2) node{$\Gamma_1$};
\draw (3,-0.5) node{$\Gamma_2$};
\end{scope}
\end{tikzpicture}
\end{center}
\caption{Topological equivalence for edge sliding} \label{figtopeq1}
\end{figure}
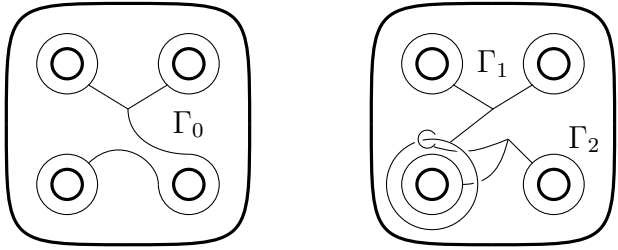

\begin{lemma} \label{lemmaslide4}
Let $\Gamma$ be an $n$--component Y--link null in $M\setminus K$. Let $J$ be a framed knot, rationally null-homologous in $M\setminus K$, and disjoint 
from $\Gamma$. Let $\Gamma'$ be obtained from $\Gamma$ be sliding an edge of $\Gamma$ along $J$ (see Figure \ref{figslide4}). 
Then $\left[(M,K);\Gamma\right]=\left[(M,K);\Gamma'\right]\ mod\ \F_{n+1}^b(M,K)$.
\end{lemma}
\begin{figure}[htb] 
\begin{center}
\begin{tikzpicture} [scale=0.3]
\newcommand{\edge}[1]{
\draw[rotate=#1] (0,0) -- (0,-1.5);
\draw[rotate=#1,dashed] (0,-1.5) -- (0,-3);}
\edge{120} \edge{240}
\draw (0,0) -- (0,-6);
\draw[->] (0,0) -- (0,-3);
\draw (1,-7) arc (0:180:1);
\draw[dashed] (-1,-7) arc (180:360:1);
\draw (-4,-4.5) arc (-90:90:1.5);
\draw[dashed] (-4,-1.5) arc (90:270:1.5);
\draw[->] (-2.5,-3.1) -- (-2.5,-2.9);
\draw (-4,-5.5) node{$J$};
\draw (1,-1) node{$\Gamma$};
\draw[->,line width=1.5pt] (6.5,-3) -- (7.5,-3);
\begin{scope} [xshift=18cm]
\edge{120} \edge{240}
\draw (0,0) -- (0,-2.5) (0,-3.5) -- (0,-6);
\draw[->] (0,0) -- (0,-2);
\draw (1,-7) arc (0:180:1);
\draw[dashed] (-1,-7) arc (180:360:1);
\draw (-4,-4.5) arc (-90:90:1.5);
\draw[dashed] (-4,-1.5) arc (90:270:1.5);
\draw (-4,-4.8) arc (-90:90:1.8);
\draw[dashed] (-4,-1.2) arc (90:270:1.8);
\draw[color=white,line width=3pt] (-2.26,-2.53) -- (-2.26,-3.48);
\draw (-2.26,-2.5) -- (0,-2.5) (-2.26,-3.5) -- (0,-3.5);
\draw (1,-1) node{$\Gamma'$};
\end{scope}
\end{tikzpicture}
\caption{Sliding an edge} \label{figslide4}
\end{center}
\end{figure}
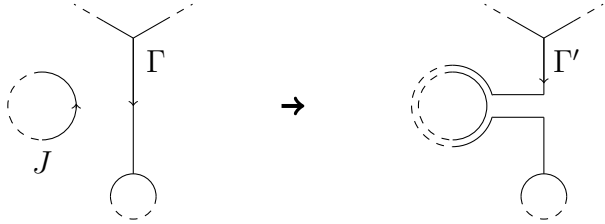
\begin{proof}
Let $\Gamma_0'$ be the component of $\Gamma'$ that contains the slided edge and let $\Gamma_0$ be the corresponding component of $\Gamma$. 
By Lemma \ref{lemmapreslide}, the surgery on $\Gamma_0'$ is equivalent to the simultaneous surgeries on $\Gamma_0$ and on a null Y--graph 
$\hat{\Gamma}_0$ which has a leaf which is a meridian of a leaf of $\Gamma_0$. 
It follows that $\left[(M,K);\Gamma\right]-\left[(M,K);\Gamma'\right]=\left[(M,K);\Gamma\cup\hat{\Gamma}_0\right]\in\F_{n+1}^b(M,K)$.
\end{proof}
In particular, the above lemma shows that the class of $\left[(M,K);\Gamma\right]\ mod\ \F_{n+1}^b(M,K)$ is invariant under full 
twists of the edges. 

\begin{lemma}[\cite{GGP} Theorem 3.1] \label{lemmaprecut}
 Let $\Gamma_0$, $\Gamma_1$, $\Gamma_2$ be the Y--graphs drawn in a genus 4 handlebody in Figure \ref{figtopeq2}. Assume this handlebody 
is embedded in a 3--manifold $V$. Then $V(\Gamma_0)\cong V(\Gamma_1\cup\Gamma_2)$.
\end{lemma}
\begin{figure}[htb] 
\begin{center}
\begin{tikzpicture} [scale=0.4]
\newcommand{\handle}{\foreach \x in {0,1,2,3} 
{\draw[very thick,rotate=90*\x] (4,0) .. controls +(0,4) and +(4,0) .. (0,4) (2,2) circle (0.5);}}
\begin{scope}
\handle
\foreach \x in {0,3} {\draw[rotate=90*\x] (2,2) circle (1);}
\draw (1.3,-1.3) -- (0.2,0) -- (1.3,1.3) (0.2,0) -- (-1,0);
\draw (-1,2) arc (0:180:1) (-3,-2) arc (-180:0:1);
\draw (-1,2) -- (-1,-2) (-3,2) -- (-3,-2);
\draw (2,0) node{$\Gamma_0$};
\end{scope}
\begin{scope} [xshift=12cm]
\handle
\foreach \x in {0,1,2,3} {\draw[rotate=90*\x] (2,2) circle (1);}
\draw (-1.3,1.3) -- (0.5,0.5) -- (1.3,-1.3);
\draw[white,line width=5pt] (2,-2) circle (1.5);
\foreach \x in {0,3} {\draw[rotate=90*\x] (2,2) circle (1.5);}
\draw (-1.3,-1.3) -- (-0.5,-0.5) -- (0.7,-1.3);
\draw (-0.5,-0.5) .. controls +(0,0.5) and +(1,0) .. (-2,0.5);
\draw (-2,3.5) arc (90:270:1.5);
\draw (-2,3.5) .. controls +(1,0) and +(-1,0) .. (0.5,2);
\draw[white,line width=5pt] (0.5,0.5) -- (1.3,1.3);
\draw (0.5,0.5) -- (1.3,1.3);
\draw (-2,-0.5) node{$\Gamma_1$};
\draw (-0.4,1.5) node{$\Gamma_2$};
\end{scope}
\end{tikzpicture}
\end{center}
\caption{Topological equivalence for leaf cutting} \label{figtopeq2}
\end{figure}
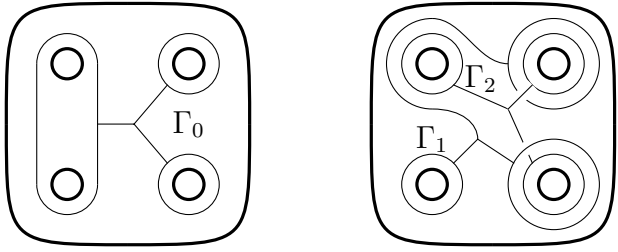

\begin{lemma} \label{lemmacut4}
Let $\Gamma$ be an $n$--component Y--link null in $M\setminus K$. Let $\ell$ be a leaf of $\Gamma$. Let $\gamma$ be a framed arc 
starting at the vertex incident to $\ell$ and ending in another point of $\ell$, embedded in $M\setminus K$ as the core of a band glued to 
the associated surface of $\Gamma$ as shown in Figure \ref{figcut4}. The arc $\gamma$ splits the leaf $\ell$ into two leaves 
$\ell'$ and $\ell''$. Denote by $\Gamma'$ (resp. $\Gamma''$) the Y--link obtained from $\Gamma$ by replacing the leaf $\ell$ by $\ell'$ (resp. $\ell''$). 
If $\ell'$ and $\ell''$ are rationally null-homologous in $M\setminus K$, then $\Gamma'$ and $\Gamma''$ 
are null Y--links and $[(M,K);\Gamma]=[(M,K);\Gamma']+[(M,K);\Gamma'']\ mod\ \F_{n+1}^b(M,K)$.
\end{lemma}
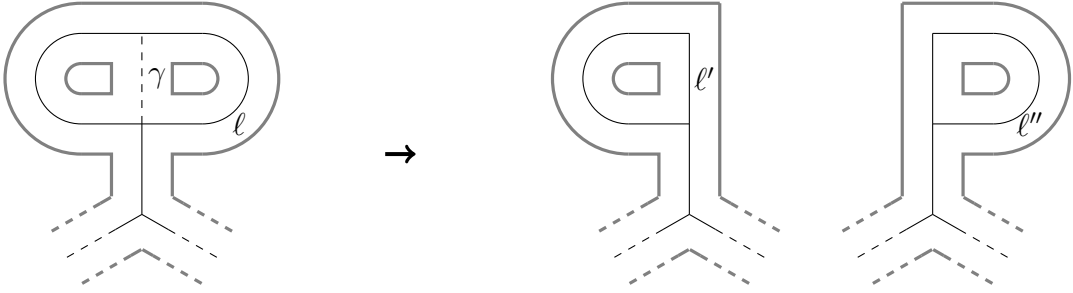
\begin{figure}[htb] 
\begin{center}
\begin{tikzpicture} [scale=0.4]
\newcommand{\edge}[1]{
\draw[rotate=#1] (0,0) -- (0,-1);
\draw[rotate=#1,gray,very thick] (1,-0.6) -- (1,-1);
\draw[rotate=#1,gray,very thick] (-1,-0.6) -- (-1,-1);
\draw[rotate=#1,dashed] (0,-1) -- (0,-3);
\draw[rotate=#1,dashed,gray,very thick] (1,-1) -- (1,-3);
\draw[rotate=#1,dashed,gray,very thick] (-1,-1) -- (-1,-3);}
\edge{60} \edge{-60}
\draw (0,0) -- (0,3);
\draw (-2,3) -- (2,3) (-2,6) -- (2,6);
\draw[dashed] (0,3) -- (0,6);
\draw (0.5,4.5) node{$\gamma$};
\draw (2,3) arc (-90:90:1.5);
\draw (-2,6) arc (90:270:1.5);
\draw (3.2,3) node{$\ell$};
\begin{scope} [gray,very thick]
\draw (2,4) -- (1,4) -- (1,5) -- (2,5);
\draw (-2,4) -- (-1,4) -- (-1,5) -- (-2,5);
\draw (2,4) arc (-90:90:0.5);
\draw (-2,5) arc (90:270:0.5);
\draw (1,0.6) -- (1,2) -- (2,2) (2,7) -- (-2,7);
\draw (-1,0.6) -- (-1,2) -- (-2,2);
\draw (2,2) arc (-90:90:2.5);
\draw (-2,7) arc (90:270:2.5);
\end{scope}
\draw[->,line width=1.5pt] (8,2) -- (9,2);
\begin{scope} [xshift=26cm]
\edge{60} \edge{-60}
\draw (0,0) -- (0,3);
\draw (0,3) -- (2,3) (0,6) -- (2,6);
\draw (0,3) -- (0,6);
\draw (2,3) arc (-90:90:1.5);
\draw (3.2,3) node{$\ell''$};
\begin{scope} [gray,very thick]
\draw (2,4) -- (1,4) -- (1,5) -- (2,5);
\draw (2,4) arc (-90:90:0.5);
\draw (1,0.6) -- (1,2) -- (2,2) (2,7) -- (-1,7);
\draw (-1,0.6) -- (-1,7);
\draw (2,2) arc (-90:90:2.5);
\end{scope}
\end{scope}
\begin{scope} [xshift=18cm]
\edge{60} \edge{-60}
\draw (0,0) -- (0,3);
\draw (-2,3) -- (0,3) (-2,6) -- (0,6);
\draw (0,3) -- (0,6);
\draw (-2,6) arc (90:270:1.5);
\draw (0.5,4.5) node{$\ell'$};
\begin{scope} [gray,very thick]
\draw (-2,4) -- (-1,4) -- (-1,5) -- (-2,5);
\draw (-2,5) arc (90:270:0.5);
\draw (1,0.6) -- (1,7) (1,7) -- (-2,7);
\draw (-1,0.6) -- (-1,2) -- (-2,2);
\draw (-2,7) arc (90:270:2.5);
\end{scope}
\end{scope}
\end{tikzpicture}
\caption{Cutting a leaf} \label{figcut4}
\end{center}
\end{figure}
\begin{proof}
Let $\Gamma_0$ (resp. $\Gamma_0'$, $\Gamma_0''$) be the component of $\Gamma$ (resp. $\Gamma'$, $\Gamma''$) that contains the leaf $\ell$ (resp. 
$\ell'$, $\ell''$). By Lemma \ref{lemmaprecut}, the surgery on $\Gamma_0$ is equivalent to simultaneous surgeries on $\Gamma_0''$ and on 
a null Y--graph $\hat{\Gamma}_0'$ obtained from $\Gamma_0'$ by sliding an edge along $\ell''$. 
Set $\hat{\Gamma}'=(\Gamma\setminus\Gamma_0)\cup\hat{\Gamma}_0'$. We have $\left[(M,K);\hat{\Gamma}'\right]+\left[(M,K);\Gamma''\right]-\left[(M,K);\Gamma\right]
=\left[(M,K);(\Gamma\setminus\Gamma_0)\cup\hat{\Gamma}_0'\cup\Gamma_0''\right]\in\F_{n+1}^b(M,K)$. 
Conclude with Lemma \ref{lemmaslide4}.
\end{proof}

The next lemma is a consequence of \cite[Lemma 4.8]{GGP}. 
\begin{lemma}
Let $\Gamma$ be an $n$--component Y--link null in $M\setminus K$. If a leaf $\ell$ of $\Gamma$ bounds a disk in $(M\setminus K)\setminus(\Gamma\setminus\ell)$ 
and has framing 1, {\em i.e.} the linking number of $\ell$ with its parallel induced by the framing of $\Gamma$ is 1, then 
$\left[(M,K);\Gamma\right]=0\ mod\ \F_{n+1}^b(M,K)$.
\end{lemma}

The above two lemmas imply that the class of $\left[(M,K);\Gamma\right]\ mod\ \F_{n+1}^b(M,K)$ does not depend on the framing of the leaves of $\Gamma$.

\begin{lemma} \label{lemmahomotopy}
 Let $\Gamma$ be an $n$--component Y--link null in $M\setminus K$. Let $\ell$ be a leaf of $\Gamma$. Assume $\Gamma\setminus\ell$ is fixed. 
Then $\left[(M,K);\Gamma\right]\ mod\ \F_{n+1}^b(M,K)$ only depends on the homotopy class of $\ell$ in $(M\setminus K)\setminus(\Gamma\setminus\ell)$.
\end{lemma}
\begin{proof}
 If the leaf $\ell$ is modified by an isotopy in $(M\setminus K)\setminus(\Gamma\setminus\ell)$, then the homeomorphism class 
of $(M,K)(\Gamma)$ is preserved. If the leaf $\ell$ crosses itself during a homotopy, apply Lemma~\ref{lemmacut4}, 
as shown in Figure \ref{figcrossing}, and conclude that $\left[(M,K);\Gamma\right]\ mod\ \F_{n+1}^b(M,K)$ is unchanged by applying Lemma \ref{lemmadisk4}.
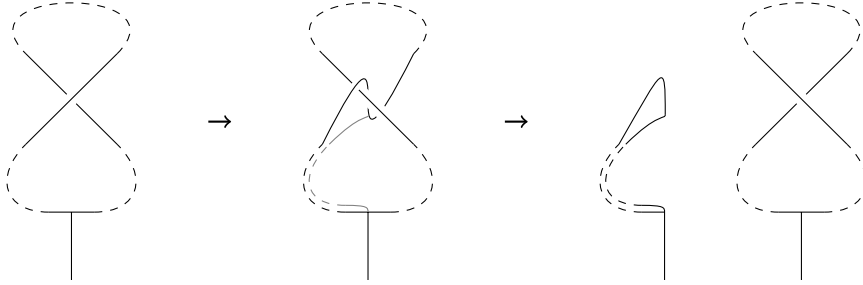
\begin{figure}[htb]
\begin{center}
\begin{tikzpicture} [scale=0.3]
\newcommand{\around}{
\draw (0,-2) -- (0,1) (-1,1) -- (1,1);
\draw[dashed] (1,1) .. controls +(2,0) and +(1.4,-1.4) .. (2,4);
\draw[dashed] (-1,1) .. controls +(-2,0) and +(-1.4,-1.4) .. (-2,4);
\draw[dashed] (2,8) .. controls +(3,3) and +(-3,3) .. (-2,8);}
\begin{scope}
\around
\draw (-2,8) -- (2,4);
\draw[white,line width=5pt] (-2,4) -- (2,8);
\draw (-2,4) -- (2,8);
\end{scope}
\draw[->,thick] (6,5) -- (7,5);
\begin{scope} [xshift=13cm]
\around
\draw (-2,8) -- (-0.7,6.7);
\draw (-2,4) .. controls +(4,8) and +(-4,-8) .. (2,8);
\draw[white,line width=5pt] (-0.5,6.5) -- (2,4);
\draw (-0.4,6.4) -- (2,4);
\draw[gray] (0,1) .. controls +(0,0.2) and +(1,0) .. (-1,1.3);
\draw[dashed,gray] (-1,1.3) .. controls +(-1.8,0) and +(-1.4,-1.4) .. (-1.7,4);
\draw[gray] (-1.7,4) .. controls +(1.4,1.4) and +(0,-0.2) .. (0.03,5.3);
\end{scope}
\draw[->,thick] (19,5) -- (20,5);
\begin{scope} [xshift=26cm]
\draw (0,-2) -- (0,1) (-1,1) -- (0,1);
\draw[dashed] (-1,1) .. controls +(-2,0) and +(-1.4,-1.4) .. (-2,4);
\draw (0,1) .. controls +(0,0.2) and +(1,0) .. (-1,1.3);
\draw[dashed] (-1,1.3) .. controls +(-1.8,0) and +(-1.4,-1.4) .. (-1.7,4);
\draw (-1.7,4) .. controls +(1.4,1.4) and +(0,-0.2) .. (0.03,5.3);
\draw (-2,4) .. controls +(2,3.5) and +(0,2.5) .. (0.03,5.3);
\end{scope}
\begin{scope} [xshift=32cm]
\around
\draw (-2,4) -- (2,8);
\draw[white,line width=5pt] (-2,8) -- (2,4);
\draw (-2,8) -- (2,4);
\end{scope}
\end{tikzpicture} \caption{Self-crossing of a leaf} \label{figcrossing}
\end{center}
\end{figure}
\end{proof}

\begin{lemma} \label{lemmahomhom}
 Let $\Gamma$ be an $n$--component Y--link null in $M\setminus K$. Let $\ell$ be a leaf of $\Gamma$. Let $\Gamma'$ be an $n$--component null Y--link 
such that $\Gamma'\setminus \ell'$ coincides with $\Gamma\setminus\ell$, where $\ell'$ is a leaf of $\Gamma'$. 
Let $\widetilde{\Gamma\setminus\ell}$ be the preimage of $\Gamma\setminus\ell$ in the infinite cyclic covering $\tilde{X}$ associated 
with $(M,K)$. Let $\tilde{\ell}$ and $\tilde{\ell'}$ be lifts of $\ell$ and $\ell'$ respectively, with the same basepoint. 
If $\ell$ and $\ell'$ are homotopic in $M\setminus K$ and $\tilde{\ell}$ and $\tilde{\ell}'$ are rationally homologous 
in $\tilde{X}\setminus(\widetilde{\Gamma\setminus\ell})$, then $[(M,K);\Gamma]=[(M,K);\Gamma']\ mod\ \F_{n+1}^b(M,K)$.
\end{lemma}
\begin{proof}
 Consider a homotopy from $\ell$ to $\ell'$ in $M\setminus K$. Thanks to Lemma \ref{lemmahomotopy}, it suffices to treat the case when 
the leaf crosses some edges or leaves of $\Gamma\setminus\ell$ during the homotopy. 
\begin{figure}[htb]
\begin{center}
\begin{tikzpicture} [scale=0.3]
\newcommand{\around}{
\draw (0,-2) -- (0,1) (-1,1) -- (1,1);
\draw[dashed] (1,1) .. controls +(2,0) and +(1.4,-1.4) .. (2,4);
\draw[dashed] (-1,1) .. controls +(-2,0) and +(-1.4,-1.4) .. (-2,4);}
\begin{scope}
\around
\draw (-2,4) .. controls +(2.5,3.5) and +(-2.5,3.5) .. (2,4);
\draw[white,line width=5pt] (-0.4,6) -- (2,6);
\draw[very thick] (-3,6) -- (-1,6) (-0.4,6) -- (3,6);
\draw[gray] (0,1) .. controls +(0,0.2) and +(1,0) .. (-1,1.3);
\draw[dashed,gray] (-1,1.3) .. controls +(-1.8,0) and +(-1.4,-1.4) .. (-1.7,4);
\draw[gray] (-1.7,4) .. controls +(1,1) and +(-1,-1) .. (0.9,5.5);
\end{scope}
\draw[->,thick] (5.5,3) -- (6.5,3);
\begin{scope} [xshift=13cm,yshift=0.1cm]
\draw[very thick] (-3,6) -- (-1.2,6) (-0.6,6) -- (3,6);
\draw (0,-2) -- (0,1) (0,1) -- (1,1);
\draw[dashed] (1,1) .. controls +(2,0) and +(1.4,-1.4) .. (2,4);
\draw (0,1) .. controls +(0,0.2) and +(1,0) .. (-1,1.3);
\draw[dashed] (-1,1.3) .. controls +(-1.8,0) and +(-1.4,-1.4) .. (-1.7,4);
\draw (-1.7,4) .. controls +(1,1) and +(-1,1) .. (2,4);
\end{scope}
\begin{scope} [xshift=12.5cm]
\draw (0,1) .. controls +(0,0.1) and +(1,0) .. (-1,1.2);
\draw[dashed] (-1,1.2) .. controls +(-1.8,0) and +(-1.4,-1.4) .. (-1.7,4);
\draw (-1.7,4) .. controls +(1,1) and +(0,-0.4) .. (0.8,5.8);
\draw (-2,4) .. controls +(1.5,2.1) and +(-0.4,0.6) .. (0.6,6.4);
\draw[dashed] (-1,1) .. controls +(-2,0) and +(-1.4,-1.4) .. (-2,4);
\draw (0,-2) -- (0,1) (-1,1) -- (0,1);
\end{scope}
\draw[->,thick] (19.5,3) -- (20.5,3);
\begin{scope} [xshift=26cm]
\draw (0,-2) -- (0,0.8) -- (-1,0.8);
\draw[dashed] (-1,0.8) .. controls +(-2.2,0) and +(-1.4,-1.4) .. (-2,4.2);
\draw (-2,4.2) .. controls +(1,1) and +(-0.1,-0.5) .. (0,5.2);
\draw (0,6) circle (0.8);
\draw[white,line width=5pt] (-0.4,6) -- (2,6);
\draw[very thick] (-3,6) -- (-1.1,6) (-0.5,6) -- (3,6);
\end{scope}
\begin{scope} [xshift=26.2cm]
\around
\draw (-2,4) .. controls +(2,1) and +(-2,1) .. (2,4);
\end{scope}
\end{tikzpicture} \caption{Crossing of an edge or a leaf} \label{figcross}
\end{center}
\end{figure}
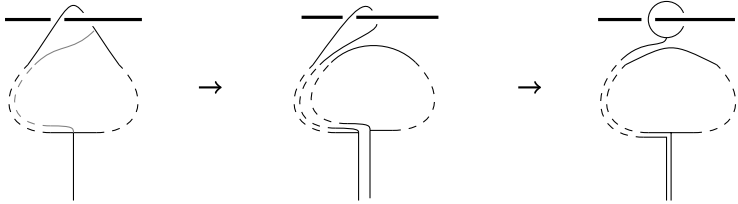
As shown in Figure \ref{figcross}, 
Lemma \ref{lemmacut4} implies that the bracket $\left[(M,K);\Gamma\right]$ is then added brackets $\left[(M,K);\hat{\Gamma}\right]$, where $\hat{\Gamma}$ 
is the null Y--link obtained from $\Gamma$ by adding the cutting arc to the edge adjacent to $\ell$, and 
by replacing $\ell$ by a meridian of the crossed edge or leaf. In the case of a meridian of an edge, 
Lemmas \ref{lemmadisk4} and \ref{lemmaslide4} show that the added bracket vanishes. 

Fix a leaf $\ell_0$ of $\Gamma\setminus \ell$. Let $\left[(M,K);\hat{\Gamma}_i\right]$, for $i\in I$, be the brackets added during the homotopy 
when the leaf $\ell$ crosses the leaf $\ell_0$. In each $\hat{\Gamma}_i$, pull the basepoint of the leaf replacing the leaf $\ell$ 
onto the initial basepoint of $\ell$. Let $\ell_i$ be the obtained leaf. Let $\tilde{\ell}_i$ be the lift of $\ell_i$ which has the 
same basepoint as $\tilde{\ell}$. Let $Y$ be the complement in $\tilde{X}$ of the preimage of $\ell_0$. In 
$H_1(Y;\Q)$, we have $\tilde{\ell}=\sum_{i\in I}\tilde{\ell}_i+\tilde{\ell}'$. 
Since $\tilde{\ell}$ and $\tilde{\ell}'$ are homologous in $\tilde{X}\setminus(\widetilde{\Gamma\setminus\ell})$, 
it implies that $\sum_{i\in I} lk_e(\tilde{\ell}_i,\tilde{\ell}_0)=0$, where $\tilde{\ell}_0$ is a lift of $\ell_0$. 
By construction of the $\tilde{\ell}_i$, each $lk_e(\tilde{\ell}_i,\tilde{\ell}_0)$ is equal to $\pm t^k$ for some $k\in\Z$. 
Thanks to Lemmas \ref{lemmadisk4}, \ref{lemmaslide4} and \ref{lemmacut4}, it follows that the 
$\hat{\Gamma}_i$ can be grouped by pairs with opposite corresponding brackets. Hence $\left[(M,K);\Gamma\right]=\left[(M,K);\Gamma'\right]\ mod\ \F_{n+1}^b(M,K)$.
\end{proof}

\begin{lemma} \label{lemmatriv4}
 Let $\Gamma$ be an $n$--component Y--link null in $M\setminus K$. Let $\ell$ be a leaf of $\Gamma$. Let $\widetilde{\Gamma\setminus\ell}$ 
be the preimage of $\Gamma\setminus\ell$ in the infinite cyclic covering $\tilde{X}$ associated with $(M,K)$. Let $\tilde{\ell}$ be a lift 
of $\ell$. If $\tilde{\ell}$ is trivial in $H_1(\tilde{X}\setminus(\widetilde{\Gamma\setminus\ell});\Q)$, 
then $[(M,K);\Gamma]=0\ mod\ \F_{n+1}^b(M,K)$.
\end{lemma}
\begin{proof}
Since $\tilde{\ell}$ has a multiple which is trivial in $H_1(\tilde{X};\Z)$, Lemma \ref{lemmacut4} allows us to assume $\tilde{\ell}$ itself 
is trivial in $H_1(\tilde{X};\Z)$. Hence $\tilde{\ell}$ is a product of commutators of loops in $\tilde{X}$. 
It follows that $\ell$ is homotopic to $\prod_{i\in I}[\alpha_i,\beta_i]$ in $M\setminus K$, where 
$I$ is a finite set and the $\alpha_i$ and $\beta_i$ satisfy $lk(\alpha_i,K)=0$, $lk(\beta_i,K)=0$. 
Construct a surface $\Sigma$ in $(M\setminus K)\setminus\Gamma$ whose handles are bands around the $\alpha_i$ and $\beta_i$, 
so that $\partial\Sigma$ is homotopic to $\ell$ in $M\setminus K$. 
Let $\Gamma'$ be the Y--link obtained from $\Gamma$ by replacing $\ell$ by $\partial \Sigma$. 
Note that the lifts of $\partial \Sigma$ are null-homologous in $\tilde{X}\setminus(\widetilde{\Gamma\setminus\ell})$. 
Hence, by Lemma \ref{lemmahomhom}, $[(M,K);\Gamma]=[(M,K);\Gamma']\ mod\ \F_{n+1}^b(M,K)$.

Let us prove that $[(M,K);\Gamma']=0\ mod\ \F_{n+1}^b(M,K)$.
Apply Lemma \ref{lemmacut4} to cut the leaf $\partial \Sigma$ into leaves $\alpha_i$, $\beta_i$, $\alpha_i^{-1}$, $\beta_i^{-1}$. 
Apply it again to reglue each leaf $\alpha_i$ with the corresponding leaf $\alpha_i^{-1}$ and each leaf $\beta_i$ with the corresponding 
leaf $\beta_i^{-1}$. The obtained Y--links all have a leaf which is homotopically trivial in the complement of $K$ and of the complement 
of the leaf in the Y--link. Then the result follows from Lemma \ref{lemmahomotopy}. 
\end{proof}

\begin{lemma} \label{lemmah4}
 Let $\Gamma$ be an $n$--component Y--link null in $M\setminus K$. Let $\ell$ be a leaf of $\Gamma$. Let $\widetilde{\Gamma\setminus\ell}$ 
be the preimage of $\Gamma\setminus\ell$ in the infinite cyclic covering $\tilde{X}$ associated with $(M,K)$. Let $\tilde{\ell}$ be a lift 
of $\ell$. Fix $\Gamma\setminus\ell$. Then the class of $[(M,K);\Gamma]\ mod\ \F_{n+1}^b(M,K)$ only depends on the class of 
$\tilde{\ell}$ in $H_1(\tilde{X}\setminus(\widetilde{\Gamma\setminus\ell});\Q)$ and this dependance is $\Q$--linear.
\end{lemma}
\begin{proof}
 Let $\Gamma'$  be a null Y--link which has a leaf $\ell'$ such that $\Gamma'\setminus\ell'$ coincides with 
$\Gamma\setminus\ell$, and $\tilde{\ell'}$ is homologous to $\tilde{\ell}$ in $\tilde{X}\setminus(\widetilde{\Gamma\setminus\ell})$, 
where $\tilde{\ell'}$ is the lift of $\ell'$ which has the same basepoint as $\tilde{\ell}$. Construct another 
null Y--link $\Gamma^\delta$ by replacing the leaf $\ell$ by $\ell-\ell'$ in $\Gamma$ (see Figure \ref{figll'4}). 
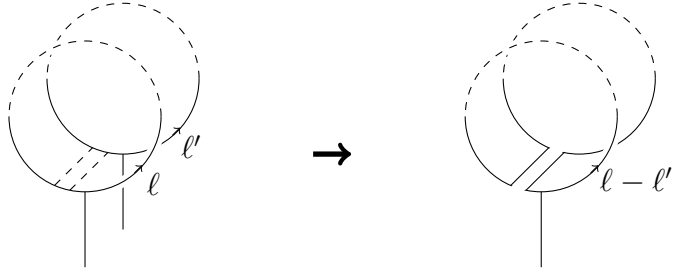
\begin{figure}[htb] 
\begin{center}
\begin{tikzpicture} [scale=0.5]
\newcommand{\leaf}[3]{
\draw[xshift=#1,yshift=#2] (0,0) -- (0,-2);
\draw[xshift=#1,yshift=#2] (-2,2) arc (-180:0:2);
\draw[xshift=#1,yshift=#2,dashed] (2,2) arc (0:180:2);
\draw[xshift=#1,yshift=#2,->] (1.3,0.5) -- (1.5,0.7);
\draw[xshift=#1,yshift=#2] (1.8,0.2) node{#3};}
\leaf{1cm}{1cm}{$\ell'$}
\draw[line width=4pt,color=white] (0,2) circle (2);
\leaf{0}{0}{$\ell$}
\draw[dashed] (-0.4,0.04) -- (0.6,1.04) (-0.8,0.17) -- (0.2,1.17);
\draw[->,line width=2pt] (6,1) -- (7,1);
\begin{scope} [xshift=12cm]
\newcommand{\sleaf}[2]{
\draw[xshift=#1,yshift=#2] (-2,2) arc (-180:0:2);
\draw[xshift=#1,yshift=#2,dashed] (2,2) arc (0:180:2);}
\sleaf{1cm}{1cm}
\draw[line width=4pt,color=white] (0,2) circle (2);
\sleaf{0}{0}
\draw (0,0) -- (0,-2);
\draw[->] (1.3,0.5) -- (1.5,0.7);
\draw (2.5,0.3) node{$\ell-\ell'$};
\draw (-0.4,0.04) -- (0.6,1.04) (-0.8,0.17) -- (0.2,1.17);
\draw[line width=4.5pt,color=white] (-0.7,0) -- (0.5,1.21);
\end{scope}
\end{tikzpicture}
\end{center}
\caption{The leaf $\ell-\ell'$} \label{figll'4}
\end{figure}
By Lemma \ref{lemmatriv4}, 
$\left[(M,K);\Gamma^\delta\right]=0\ mod\ \F_{n+1}^b(M,K)$. Thus Lemma \ref{lemmacut4} implies $\left[(M,K);\Gamma\right]=\left[(M,K);\Gamma'\right]\ mod\ \F_{n+1}^b(M,K)$. 
Linearity follows from Lemma \ref{lemmacut4}. 
\end{proof}

   \section{Colored diagrams and Y--links} \label{secdiagrams}

In this section, we apply clasper calculus to obtain the maps from diagrams spaces to graded quotients of Proposition \ref{propphinintro} and Theorem \ref{thphiZ}. 

Fix a Blanchfield module $(\Al,\bl)$. An $(\Al,\bl)$--colored diagram is an {\em elementary ($(\Al,\bl)$--colored) diagram} if its edges 
that connect two trivalent vertices are colored by powers of $t$ and if its edges adjacent to univalent vertices are colored by $1$. 
Below, we associate a null Y--link with some elementary diagrams that generate $\tilde{\A}_n(\Al,\bl)$. 
Let $(M,K)\in\Ens(\Al,\bl)$. Let $\xi: (\Al,\bl) \to (\Al,\bl)(M,K)$ be an isomorphism. Let $m(K)$ be a meridian of $K$. 

Let $D$ be an elementary diagram. An embedding of $D$ in $M\setminus K$ is {\em admissible} if the following conditions are satisfied.
\begin{itemize}
 \item The vertices of $D$ are embedded in some ball $B\subset M\setminus K$.
 \item Consider an edge colored by $t^k$. The homology class in $H_1(M\setminus K;\Z)$ of the closed curve obtained by connecting 
  the extremities of this edge by a path in $B$ is $k\,m(K)$.
\end{itemize}
Such an embedding always exists. It suffices to embed the diagram in $B$, and to let each edge colored by $t^k$ turn around $K$, $k$ times. 
With an admissible embedding of an elementary diagram, we wish to associate a null Y--link. 

Let $\Gamma$ be a Y--graph null in $M\setminus K$. Let $p$ be the internal vertex of $\Gamma$. 
Let $\ell$ be a leaf of $\Gamma$. The curve $\hat{\ell}$ drawn in Figure \ref{figextension} is the {\em extension of $\ell$ in $\Gamma$}. 
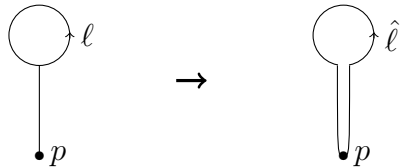
\begin{figure}[htb]
\begin{center}
\begin{tikzpicture} [scale=0.4]
\begin{scope} [xshift=-5cm]
\draw (0,0) -- (0,3);
\draw (0,4) circle (1);
\draw[->] (1,3.9) -- (1,4.1);
\draw (0,0) node[right] {$p$};
\draw (0,0) node {$\scriptstyle{\bullet}$};
\draw (1,4) node[right] {$\ell$};
\end{scope}
\draw[->,very thick] (-0.5,2.5) -- (0.5,2.5);
\begin{scope} [xshift=5cm]
\draw (0,0) .. controls +(0.2,0) and +(0,-3) .. (0.2,3);
\draw (0,0) .. controls +(-0.2,0) and +(0,-3) .. (-0.2,3);
\draw (0,4) circle (1);
\draw[white,very thick] (-0.18,3) -- (0.18,3);
\draw[->] (1,3.9) -- (1,4.1);
\draw (0,0) node[right] {$p$};
\draw (0,0) node {$\scriptstyle{\bullet}$};
\draw (1,4) node[right] {$\hat{\ell}$};
\end{scope}
\end{tikzpicture} 
\end{center}
\caption{Extension of a leaf in a Y--graph} \label{figextension}
\end{figure}

Let $D$ be an elementary diagram, equipped with an admissible embedding in $M\setminus K$. 
Equip $D$ with the framing induced by an immersion in the plane which induces the fixed orientation of the trivalent vertices. 
If an edge connects two trivalent vertices, then insert a little Hopf link in this edge, as shown in Figure \ref{figrempla4}. 
\begin{figure}[htb] 
\begin{center}
\begin{tikzpicture} [scale=0.2]
\draw (-36,0) -- (-18,0);
\draw (-36,0) node{$\scriptscriptstyle{\bullet}$};
\draw (-18,0) node{$\scriptscriptstyle{\bullet}$};
\draw[->,line width=1.5pt,>=latex] (-10.5,0) -- (-7.5,0);
\draw (0,0) node{$\scriptscriptstyle{\bullet}$};
\draw (0,0) -- (6,0);
\draw (8,0) circle (2);
\draw[color=white,line width=6pt] (8,0) arc (-180:-90:2);
\draw (10,0) circle (2);
\draw[color=white,line width=6pt] (10,0) arc (0:90:2);
\draw (10,0) arc (0:90:2);
\draw[->] (11.5,1.3) -- (11.3,1.5);
\draw[->] (6.5,-1.3) -- (6.7,-1.5);
\draw (12,0) -- (18,0);
\draw (18,0) node{$\scriptscriptstyle{\bullet}$};
\end{tikzpicture}
\end{center}
\caption{Replacement of an edge} \label{figrempla4}
\end{figure}
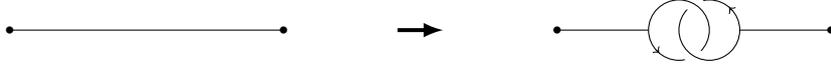
At each univalent vertex $v$, 
glue a leaf $\ell_v$, trivial in $H_1(M\setminus K;\Q)$, in order to obtain a null Y--link $\Gamma$. Let $V$ be the set of all univalent 
vertices of $D$. Let $B$ be the ball in the definition of the admissible embedding of $D$. Let $\tilde{B}$ be a lift of $B$ in the infinite 
cyclic covering $\tilde{X}$ of the exterior of $K$ in $M$. 
For $v\in V$, let $\gamma_v$ be the coloring of $v$, let $\hat{\ell}_v$ be the extension of $\ell_v$ in $\Gamma$ 
and let $\tilde{\ell}_v$ be the lift of $\hat{\ell}_v$ in $\tilde{X}$ defined by lifting the basepoint in $\tilde{B}$. 
The null Y--link $\Gamma$ is a {\em realization of $D$ in $(M,K)$ with respect to $\xi$} if the following conditions are satisfied:
\begin{itemize}
 \item for all $v\in V$, $\tilde{\ell}_v$ is homologous to $\xi(\gamma_v)$, 
 \item for all $(v,v')\in V^2$, $lk_e(\tilde{\ell}_v,\tilde{\ell}_{v'})=f_{vv'}$. 
\end{itemize}
If such a realization exists, the elementary diagram $D$ is {\em $\xi$--realizable}. 

\begin{lemma} \label{lemmaind}
 Let $(M,K)\in\Ens(\Al,\bl)$. Let $\xi: (\Al,\bl)\to(\Al,\bl)(M,K)$ be an isomorphism. 
Let $D\in\tilde{\A}_n(\Al,\bl)$ be an elementary diagram of degree $n>0$, $\xi$--realizable. Let $\Gamma$ be a realization of $D$ in $(M,K)$ with respect to $\xi$. 
Then the class of $[(M,K);\Gamma]\ mod\ \F_{n+1}^b(M,K)$ does not depend on the realization of $D$. 
\end{lemma}
\begin{proof}
 If the ball $B$ and its lift $\tilde{B}$ are fixed, then the result follows 
from Lemmas \ref{lemmaslide4} and \ref{lemmah4}. Fix the ball $B$ and consider another lift $\tilde{B}'=\tau^k(\tilde{B})$ of $B$, 
where $\tau$ is the automorphism of $\tilde{X}$ which induces the action of $t$ and $k\in\Z$. A realization of $D$ with respect to $\tilde{B}'$ 
can be obtained from $\Gamma$ by letting the internal vertex of each Y--graph in $\Gamma$ turn around $K$, $k$ times, and come back into $B$, 
by an isotopy of $(M,K,\Gamma)$. 
This does not change the result of the surgeries on these Y--graphs, hence this does not modify the bracket $[(M,K);\Gamma]$. 
Now consider two balls $B_1$ and $B_2$ in $M\setminus K$. If $B_1\subset B_2$, a realization of $D$ with respect to $B_1$ is a realization 
of $D$ with respect to $B_2$. If $B_1\cap B_2\neq \emptyset$, there is a ball $B_3\subset B_1\cap B_2$. If $B_1\cap B_2=\emptyset$, 
there is a ball $B_3\supset B_1\cup B_2$. Hence the class of the bracket $[(M,K);\Gamma]$ does not depend on the chosen ball $B$. 
\end{proof}
In the sequel, if $D$ is a $\xi$--realizable elementary diagram, $[(M,K);D]_\xi$ denotes the class of $[(M,K);\Gamma]$ modulo $\F_{n+1}^b(M,K)$. 

Let $D$ be any elementary diagram. Let $V$ be the set of all univalent vertices of $D$. For any family of rational numbers $(q_v)_{v\in V}$, 
define an elementary diagram $D'=(q_v)_{v\in V}\cdot D$ from $D$ in the following way. Keep the same graph and the same colorings of the edges. 
For $v\in V$, multiply the coloring of $v$ by $q_v$. For $v\neq v'\in V$, set $f_{vv'}^{D'}=q_vq_{v'}f_{vv'}^D$. 
\begin{lemma}
 Let $(M,K)\in\Ens(\Al,\bl)$. Let $\xi: (\Al,\bl) \to (\Al,\bl)(M,K)$ be an isomorphism. 
Let $D$ be any elementary diagram. Let $V$ be the set of all univalent vertices of $D$. 
Then there exists a family of positive integers $(s_v)_{v\in V}$ such that $(s_v)_{v\in V}\cdot D$ is $\xi$--realizable.
\end{lemma}
\begin{proof}
 Let $\tilde{X}$ be the infinite cyclic covering associated with $(M,K)$. 
Since any homology class in $\Al$ has a multiple which can be represented by a knot in $\tilde{X}$, we can assume that 
the color $\gamma_v$ of each vertex $v$ in $V$ can be represented by a knot in $\tilde{X}$. From $D$, define as above a Y--link $\Gamma$, 
null in $M\setminus K$, with leaves $\ell_v$ which satisfy the condition that $\tilde{\ell}_v$ is homologous to $\xi(\gamma_v)$ 
for all $v\in V$. For $v\neq v'\in V$, set $P_{vv'}=lk_e(\tilde{\ell}_v,\tilde{\ell}_{v'})-f_{vv'}$. We can assume that $P_{vv'}\in\Z[t^{\pm1}]$ 
for all $v\neq v'\in V$. Add well chosen meridians of $\ell_v$ to $\ell_{v'}$ to get $P_{vv'}=0$. 
\end{proof}

\begin{lemma} \label{lemmarationalmultiples}
 Let $(M,K)\in\Ens(\Al,\bl)$. Let $\xi: (\Al,\bl)\to(\Al,\bl)(M,K)$ be an isomorphism. Let $D$ be an elementary $(\Al,\bl)$--colored diagram. 
Let $V$ be the set of all univalent vertices of $D$. Let $(s_v)_{v\in V}$ and $(s'_v)_{v\in V}$ be families of integers 
such that $(s_v)_{v\in V}\cdot D$ and $(s'_v)_{v\in V}\cdot D$ are $\xi$--realizable. Then:
$$\prod_{v\in V} s'_v \left[(M,K);(s_v)_{v\in V}\cdot D\right]_\xi=\prod_{v\in V} s_v \left[(M,K);(s'_v)_{v\in V}\cdot D\right]_\xi.$$
\end{lemma}
\begin{proof}
 Let $\Gamma$ be a realization of $(s_v)_{v\in V}\cdot (s'_v)_{v\in V}\cdot D$ in $(M,K)$ with respect to $\xi$. 
By Lemma~\ref{lemmah4}, $\left[(M,K);\Gamma\right]$ is equal to both sides of the equality. 
\end{proof}

Let $D$ be an elementary $(\Al,\bl)$--colored diagram. Let $V$ be the set of all univalent vertices of $D$. 
The above result allows us to define: $$\left[(M,K);D\right]_\xi=\prod_{v\in V}\frac{1}{s_v}\left[(M,K);(s_v)_{v\in V}\cdot D\right]_\xi\ \in\G_n^b(M,K),$$ 
where $(s_v)_{v\in V}$ is any family of integers such that $(s_v)_{v\in V}\cdot D$ is $\xi$--realizable. 

\begin{lemma} \label{lemmaind2}
 Let $D$ be an elementary $(\Al,\bl)$--colored diagram. Let $(M,K)$ and $(M',K')$ be $\Q$SK--pairs in $\Ens(\Al,\bl)$. 
Let $\xi: (\Al,\bl)\to(\Al,\bl)(M,K)$ and $\xi': (\Al,\bl)\to(\Al,\bl)(M',K')$ be isomorphisms. 
Then $\left[(M',K');D\right]_{\xi'}=\left[(M,K);D\right]_\xi\ mod\ \F_{n+1}(\Al,\bl)$.
\end{lemma}
\begin{proof}
 Set $\zeta=\xi'\circ\xi^{-1}$. By Theorem \ref{thM3}, $(M',K')$ can be obtained from $(M,K)$ by a finite sequence of null LP--surgeries 
which induces $\zeta\circ m_k$, for $k\in\Z$, where $m_k$ is the multiplication by $t^k$. 
Assume the sequence contains a single surgery $\left(\frac{A'}{A}\right)$. Let $V$ be the set of all univalent vertices of $D$. 
Let $(s_v)_{v\in V}$ be a family of integers such that $(s_v)_{v\in V}\cdot D$ 
is $\xi$--realizable by a null Y--link $\Gamma$ in $(M\setminus K)\setminus A$. Then:
$$\left[(M,K);\Gamma,\frac{A'}{A}\right]=\left[(M,K);\Gamma\right]-\left[(M',K');\Gamma\right].$$
In $(M',K')$, $\Gamma$ is a realization of $(s_v)_{v\in V}\cdot D$ with respect to $\xi'\circ m_k$. Hence it is also a realization of $(s_v)_{v\in V}\cdot D$ 
with respect to $\xi'$ (it suffices to change the lift $\tilde{B}$ of the ball $B$, see Lemma \ref{lemmaind}). 

The case of several surgeries easily follows.
\end{proof}

In the sequel, the class of $[(M,K);D]_\xi$ modulo $\F_{n+1}(\Al,\bl)$ is denoted by $[D]$.

\begin{proposition} \label{propphin}
 Fix a Blanchfield module $(\Al,\bl)$. Let $n> 0$. 
There is a canonical, $\Q$--linear and surjective map $\varphi_n : \A_n(\Al,\bl) \twoheadrightarrow \G_n^b(\Al,\bl)$ given by 
$D\mapsto [D]$ for any elementary diagram $D$.
\end{proposition}
\begin{proof}
Let $\mathcal{D}_n$ be the rational vector space freely generated by the $(\Al,\bl)$--colored diagrams of degree $n$. 
If $D$ is an elementary $(\Al,\bl)$--colored diagram, set $\tilde{\varphi}_n(D)=[D]$. 
Define $\tilde{\varphi}_n(D)$ for any $(\Al,\bl)$--colored diagram $D$ so that the obtained $\Q$--linear map 
$\tilde{\varphi}_n:\mathcal{D}_n\to\G_n^b(\Al,\bl)$ satisfies the relation LE and EV. Let us check that 
$\tilde{\varphi}_n$ satisfies the relations AS, IHX, OR, Hol, LV, LD and Aut. 
OR is trivial. LV follows from Lemma \ref{lemmah4}. Hol is obtained by letting the corresponding 
vertex of a realization of $D$ turn around the knot $K$. AS and IHX respectively follow from \cite[Corollary 4.6]{GGP} and \cite[Lemma 4.10]{GGP}. 
Aut follows from Lemma \ref{lemmaind2}. For the relation LD, it suffices to prove that 
$\tilde{\varphi}_n(D)=\tilde{\varphi}_n(D')+\tilde{\varphi}_n(D_0)$, where $D$, $D'$ and $D_0$ are elementary diagrams identical, 
except for the part drawn in Figure \ref{figdiag1}. Note that the edges adjacent to $v_1$ and $v_2$ are colored by~$1$. 
Since the diagram $D_0$ and the diagram $D_0'$ drawn in Figure \ref{figdiag2} can be realized by the same null Y--link, we have 
$\tilde{\varphi}_n(D_0)=\tilde{\varphi}_n(D_0')$. To see that $\tilde{\varphi}_n(D_0')=\tilde{\varphi}_n(D_0'')$, 
apply the relation LV at the vertex $v_2$ to obtain $\tilde{\varphi}_n(D_0'')=\tilde{\varphi}_n(D_0')+\tilde{\varphi}_n(D_{00})$ 
and then apply the relation LV at the vertex $v_1$ to obtain $\tilde{\varphi}_n(D_{00})=0$. 
Apply again the relation LV at the vertex $v_1$ to get $\tilde{\varphi}_n(D)=\tilde{\varphi}_n(D')+\tilde{\varphi}_n(D_0'')$. 
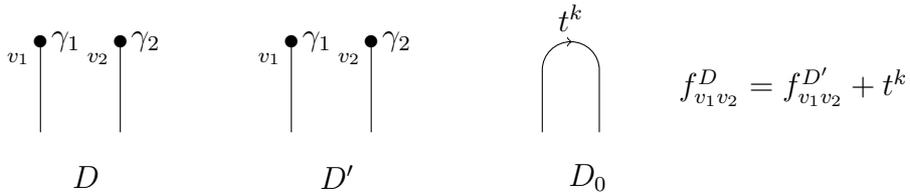
\begin{figure}[htb] 
\begin{center}
\begin{tikzpicture} [scale=0.3]
\draw (-3,0) -- (-3,4);
\draw (-3,4) node{$\bullet$};
\draw (-3,4) node[below left] {$\scriptstyle{v_1}$};
\draw (-3,4) node[right] {$\gamma_1$};
\draw (0.5,0) -- (0.5,4);
\draw (0.5,4) node{$\bullet$};
\draw (0.5,4) node[below left] {$\scriptstyle{v_2}$};
\draw (0.5,4) node[right] {$\gamma_2$};
\draw (-1,-2) node {$D$};
\draw (8,0) -- (8,4);
\draw (8,4) node{$\bullet$};
\draw (8,4) node[below left] {$\scriptstyle{v_1}$};
\draw (8,4) node[right] {$\gamma_1$};
\draw (11.5,0) -- (11.5,4);
\draw (11.5,4) node{$\bullet$};
\draw (11.5,4) node[below left] {$\scriptstyle{v_2}$};
\draw (11.5,4) node[right] {$\gamma_2$};
\draw (10,-2) node {$D'$};
\draw (19,0) -- (19,2.75) (21.5,0) -- (21.5,2.75) arc (0:180:1.25);
\draw[->] (20.2,4) -- (20.3,4);
\draw (20.25,4) node[above] {$t^k$};
\draw (21,-2) node {$D_0$};
\draw (30,2) node {$f^D_{v_1v_2}=f^{D'}_{v_1v_2}+t^k$};
\end{tikzpicture}
\end{center}
\caption{The diagrams $D$, $D'$, and $D_0$, where $\gamma_1,\gamma_2\in\Al$ and $k\in\Z$.} \label{figdiag1}
\end{figure}
\begin{figure}[htb] 
\begin{center}
\begin{tikzpicture} [xscale=0.25,yscale=0.3]
\draw (-3,0) -- (-3,4);
\draw (-3,4) node{$\bullet$};
\draw (-3,4) node[below left] {$\scriptstyle{v_1}$};
\draw (-3,4) node[right] {$0$};
\draw (0.5,0) -- (0.5,4);
\draw (0.5,4) node{$\bullet$};
\draw (0.5,4) node[below left] {$\scriptstyle{v_2}$};
\draw (0.5,4) node[right] {$0$};
\draw (4,-3) node {$D_0'$};
\draw (7,2) node {\scalebox{0.9}{$\begin{array}{l} f_{v_1v_2}(t)=t^k \\ f_{v_1v}(t)=0 \\ f_{v_2v}(t)=0 \end{array}$}};
\begin{scope} [xshift=20cm]
\draw (-3,0) -- (-3,4);
\draw (-3,4) node{$\bullet$};
\draw (-3,4) node[below left] {$\scriptstyle{v_1}$};
\draw (-3,4) node[right] {$0$};
\draw (0.5,0) -- (0.5,4);
\draw (0.5,4) node{$\bullet$};
\draw (0.5,4) node[below left] {$\scriptstyle{v_2}$};
\draw (0.5,4) node[right] {$\gamma_2$};
\draw (4,-3) node {$D_0''$};
\draw (8,2) node {\scalebox{0.9}{$\begin{array}{l} f_{v_1v_2}(t)=t^k \\ f_{v_1v}(t)=0 \\ f_{v_2v}(t)=f^{D}_{v_2v}(t) \end{array}$}};
\end{scope}
\begin{scope} [xshift=42cm]
\draw (-3,0) -- (-3,4);
\draw (-3,4) node{$\bullet$};
\draw (-3,4) node[below left] {$\scriptstyle{v_1}$};
\draw (-3,4) node[right] {$0$};
\draw (0.5,0) -- (0.5,4);
\draw (0.5,4) node{$\bullet$};
\draw (0.5,4) node[below left] {$\scriptstyle{v_2}$};
\draw (0.5,4) node[right] {$\gamma_2$};
\draw (4,-3) node {$D_{00}$};
\draw (8.5,2) node {\scalebox{0.9}{$\begin{array}{l} f_{v_1v_2}(t)=0 \\ f_{v_1v}(t)=0 \\ f_{v_2v}(t)=f^{D}_{v_2v}(t) \end{array}$}};
\end{scope}
\end{tikzpicture}
\end{center}
\caption{The diagrams $D_0'$, $D_0''$ and $D_{00}$, with $v\neq v_1,v_2$.} \label{figdiag2}
\end{figure}
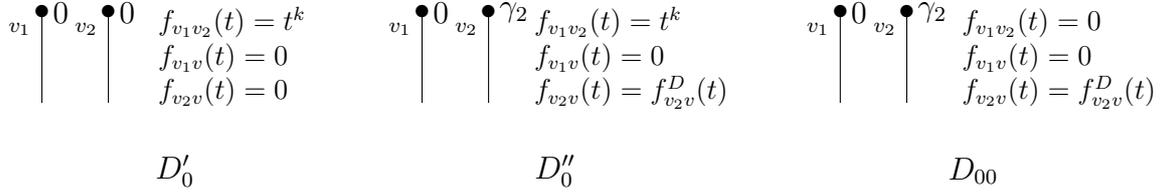

Finally, the map $\tilde{\varphi}_n$ induces a canonical $\Q$--linear map $\varphi_n:\A_n(\Al,\bl)\to\G_n^b(\Al,\bl)$. 
For $(M,K)\in\Ens(\Al,\bl)$, any $n$--component Y--link null in $M\setminus K$ is a realization of an 
elementary $(\Al,\bl)$--colored diagram, which is the disjoint union of $n$ diagrams of degree 1. Hence $\varphi_n$ is surjective. 
\end{proof}

Now that the map $\varphi_n$ is well-defined, we can prove the second point of Theorem \ref{thinvariantZ}.
\begin{proof}[Proof of the second statement of Theorem \ref{thinvariantZ}]
 Take an $(\Al,\bl)$--colored diagram of degree $n$. Let $\Gamma=\Gamma_1\sqcup\dots\sqcup\Gamma_n$ be a realization of $D$ in some $\Q$SK--pair 
 $(M,K)\in\Ens(\Al,\bl)$. For each $i\in\{1,\dots,n\}$, fix a lift $\tilde{\Gamma}_i$ of $\Gamma_i$ in the infinite cyclic covering $\tilde{X}$ 
 of $M\setminus K$, and represent it schematically as
 
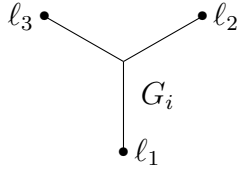
\begin{center}
 \begin{tikzpicture} [scale=0.15]
 \foreach \o in {0,120,240} {\draw[rotate=\o] (0,0) -- (0,-8) node {$\scriptstyle{\bullet}$};}
 \draw (3,-3) node{$G_i$};
 \draw (0,-8) node[right] {$\ell_1$};
 \draw (6.9,4) node[right] {$\ell_2$};
 \draw (-6.9,4) node[left] {$\ell_3$};
 \end{tikzpicture}
 \end{center}
 where $\ell_1,\ell_2,\ell_3$ are the leaves of $\tilde{\Gamma}_i$. 
 By \cite[Theorem 1.1]{Les3} for the Lescop invariant and \cite[Theorem 1.1]{M6} for the Kricker invariant, 
 the image by $Z$ of the bracket $[(M,K);\Gamma]$ is, modulo $\F_{n+1}(\delta)$, the sum of all diagrams obtained from 
 $G=\sqcup_{1\leq i\leq n}G_i$ by pairwise gluing all univalent vertices as follows ---note that the choice of the lifts of the $\Gamma_i$ 
 has no importance thanks to the relation Hol.
 
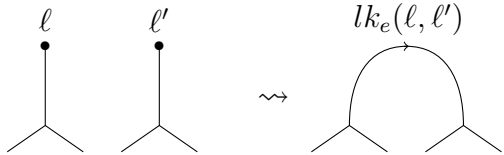
\begin{center}
 \begin{tikzpicture} [yscale=0.7]
 \foreach \x in {0,1.5} {
 \draw (\x-0.5,0) -- (\x,0.5) -- (\x,2) node {$\scriptstyle{\bullet}$} (\x+0.5,0) -- (\x,0.5);}
 \draw (0,2.5) node {$\ell$} (1.5,2.5) node {$\ell'$} (3,1) node {$\rightsquigarrow$};
 \draw (4.5,0) -- (4,0.5) (3.5,0) -- (4,0.5) .. controls +(0,0.9) and +(-0.5,0) .. (4.75,2) .. controls +(0.5,0) and +(0,0.9) .. (5.5,0.5) -- (5,0) (5.5,0.5) -- (6,0);
 \draw[->] (4.72,2) -- (4.78,2) node[above] {$lk_e(\ell,\ell')$};
 \end{tikzpicture}
 \end{center}
 When an edge of $D$ joins two trivalent vertices, then the corresponding two univalent vertices in $G$ are labelled by curves $\ell$ 
 and $\ell'$ such that the equivariant linking of $\ell$ is 0 with any curve labelling a vertex of $G$ other than $\ell'$, and reciprocally. 
 Moreover, relevant choices of the lifts of the $\Gamma_i$ ensure that $lk_e(\ell,\ell')=1$. Finally, modulo $\F_{n+1}(\delta)$, we have
 $Z([(M,K);\Gamma])=\psi_n(D)$. Hence $Z_k([(M,K);\Gamma])=0$ if $k<n$ and $Z_n\circ\varphi_n(D)=\psi_n(D)$. 
\end{proof}

In the setting of $\Z$SK--pairs, all the results of Section \ref{secborro} apply since we work modulo $\F_{n+1}^b(M,K)$. 
All the results of the current section apply as well. In Lemma \ref{lemmaind2}, note that we use Theorem \ref{thM3Z} instead of Theorem \ref{thM3}. 
We finally get the similar result to the above proposition. 
\begin{proposition} \label{propphinZ}
 Fix an integral Blanchfield module $(\Al_\Z,\bl_\Z)$. Let $n> 0$. 
There is a canonical, $\Q$--linear and surjective map $\varphi_n^\Z : \A_n^\Z(\Al_\Z,\bl_\Z) \twoheadrightarrow \G_n^\Z(\Al_\Z,\bl_\Z)$ 
given by $D\mapsto [(M,K);D]_\xi$ for any elementary diagram $D$, where $(M,K)$ is any $\Z$SK--pair with Blanchfield module $(\Al_\Z,\bl_\Z)$ and 
$\xi: (\Al_\Z,\bl_\Z) \to (\Al_\Z,\bl_\Z)(M,K)$ is any isomorphism.
\end{proposition}

Fix $(\Al_\Z,\bl_\Z)$ and set $(\Al,\bl)=(\Q\otimes_\Z\Al_\Z,Id_\Q\otimes\bl_\Z)$. 
The corresponding map $\varphi_n$ satisfies $\varphi_n\circ p_n=\omega_n\circ\varphi_n^\Z$, where $p_n:\A^\Z_n(\Al,\bl)\twoheadrightarrow\A_n(\Al,\bl)$ 
is the natural projection and $\omega_n: \G_n^\Z(\Al_\Z,\bl_\Z) \to \G_n^b(\Al,\bl)$ is the map induced by the inclusion $\F_n^\Z(\Al_\Z,\bl_\Z) \hookrightarrow \F_n(\Al,\bl)$.

    \section{The surjective map $\varphi_n : \A_n^{aug}(\Al,\bl)\to\G_n(\Al,\bl)$} \label{secsurjective}

In this section, we prove Theorems \ref{thextend} and \ref{thG1}.

Fix a Blanchfield module $(\Al,\bl)$. 
Let $(M,K)$ be a $\Q$SK--pair in $\Ens(\Al,\bl)$. Let $\xi:(\Al,\bl)\to(\Al,\bl)(M,K)$ be an isomorphism. 
Let $n>0$. Let $D$ be an $(\Al,\bl)$--augmented diagram of degree $n$ whose Jacobi part $D_J$ is elementary. 
With an isolated vertex colored by a prime integer $p$, we associate a surgery $\frac{B_p}{B^3}$, where $B_p$ is a fixed $\Q$--ball 
such that $|H_1(B_p;\Z)|=p$. Hence, if $D_J$ is $\xi$--realizable, with a realization of $D_J$, we associate a family of $n$ disjoint null 
LP--surgeries. 
\begin{lemma} \label{lemmadiag}
 Let $(M,K)\in\Ens(\Al,\bl)$. Let $\xi:(\Al,\bl)\to(\Al,\bl)(M,K)$ be an isomorphism. Let $n>0$. Let $D$ be an $(\Al,\bl)$--augmented diagram 
whose Jacobi part $D_J$ is elementary. Let $(p_i)_{1\leq i\leq n-k}$ be the labels of the isolated vertices 
of $D$. If $D_J$ is $\xi$--realizable, let $\Gamma$ be a realization of $D_J$ in $(M,K)$ with respect to $\xi$. 
Then:
\begin{itemize}
 \item $\left[(M,K);D\right]_\xi:=\left[(M,K);\left(\frac{B_{p_i}}{B^3}\right)_{1\leq i\leq n-k},\Gamma\right]\in\G_n(\Al,\bl)$ does not depend on $(M,K)$, 
  on $\xi$ or on the realization $\Gamma$ of $D_J$.
\end{itemize}
If $D_J$ is any elementary diagram, set $[D]=\prod_{v\in V}\frac{1}{s_v}\left[(M,K);(s_v)_{v\in V}\cdot D\right]_\xi$, where $(s_v)_{v\in V}$ 
is a family of integers such that $(s_v)_{v\in V}\cdot D_J$ is $\xi$--realizable and $(s_v)_{v\in V}\cdot D$ is the disjoint union of $(s_v)_{v\in V}\cdot D_J$ 
with the $0$--valent part of $D$. Then:
\begin{itemize}
 \item $[D]\in\G_n(\Al,\bl)$ does not depend on $(s_v)_{v\in V}$, $(M,K)$ or $\xi$. 
\end{itemize}
\end{lemma}
\begin{proof}
Take $(M',K')\in\Ens(\Al,\bl)$ and an isomorphism $\xi':(\Al,\bl)\to(\Al,\bl)(M',K')$ such that $D_J$ is $\xi'$--realizable. Let $\Gamma'$ be a realization of $D_J$ 
with respect to $\xi'$. By Proposition \ref{propphin}:
$$\left[(M',K');\Gamma'\right]=\left[(M,K);\Gamma\right]\ mod\ \F_{k+1}(\Al,\bl).$$
Let $p$ be a prime integer. Let $M_p=B^3\cup_{\partial B^3}B_p$. In the equality in $\F_0(\Al,\bl)$ corresponding to the above relation, 
make a connected sum of each $\Q$SK--pair with $M_p$. Then substract the new equality from the original one, to obtain:
$$\left[(M',K');\frac{B_p}{B^3},\Gamma'\right]=\left[(M,K);\frac{B_p}{B^3},\Gamma\right]\ mod\ \F_{k+2}(\Al,\bl).$$
Applying this process $n-k$ times, we get:
$$\left[(M',K');\left(\frac{B_{p_i}}{B^3}\right)_{1\leq i\leq n-k},\Gamma'\right]=\left[(M,K);\left(\frac{B_{p_i}}{B^3}\right)_{1\leq i\leq n-k},\Gamma\right]\ mod\ \F_{n+1}(\Al,\bl).$$

If $D_J$ is any elementary diagram, use Lemma \ref{lemmah4}, as in Lemma \ref{lemmarationalmultiples}, to prove that 
$\left[(M,K);D\right]_\xi=\prod_{v\in V}\frac{1}{s_v}\left[(M,K);(s_v)_{v\in V}\cdot D\right]_\xi$ does not depend on the 
family of integers $(s_v)_{v\in V}$ such that $(s_v)_{v\in V}\cdot D_J$ is $\xi$--realizable. 
Conclude with the first assertion of the lemma.
\end{proof}

The above result implies that the map $\varphi_n : \A_n(\Al,\bl) \twoheadrightarrow \G_n^b(\Al,\bl)$ extends to a canonical $\Q$--linear map 
$\varphi_n : \A_n^{aug}(\Al,\bl) \to \G_n(\Al,\bl)$ defined by $\varphi_n(D)=[D]$ for any diagram $D\in\A_n^{aug}(\Al,\bl)$ 
whose Jacobi part is elementary. To prove Theorem \ref{thextend}, it remains 
to show that the map $\varphi_n : \A_n^{aug}(\Al,\bl) \to \G_n(\Al,\bl)$ is surjective. We first recall results from \cite{M2} 
and give consequences of them.
\begin{definition}
 Let $d$ be a positive integer. A \emph{$d$--torus} is a $\Q$--torus $T_d$ such that: 
\begin{itemize}
 \item $H_1(\partial T_d;\Z)=\Z \alpha \oplus \Z\beta$, with algebraic intersection number $\langle\alpha,\beta\rangle=1$,
 \item $d\alpha=0$ in $H_1(T_d;\Z)$,
 \item $\beta=d\gamma$ in $H_1(T_d;\Z)$, where $\gamma$ is a curve in $T_d$,
 \item $H_1(T_d;\Z)=\frac{\Z}{d\Z} \alpha \oplus \Z \gamma$.
\end{itemize}
\end{definition}
\begin{definition}
 An \emph{elementary surgery} is an LP--surgery among the following ones:
\begin{enumerate}
 \item connected sum (genus 0),
 \item LP--replacement of a standard torus by a $d$--torus (genus 1),
 \item Borromean surgery (genus 3).
\end{enumerate}
\end{definition}
The next result generalizes the similar result of Habegger \cite{Habe} and Auclair and Lescop \cite{AL} for $\Z$--handlebodies and Borromean surgeries. 
\begin{theorem}[\cite{M2} Theorem 1.15] \label{thdecsur}
 If $A$ and $B$ are two $\Q$--handlebodies with LP--identified boundaries, then $B$ can be obtained from $A$ 
by a finite sequence of elementary surgeries and their inverses in the interior of the $\Q$--handlebodies.
\end{theorem}

\begin{corollary} \label{cordecel}
 The space $\F_n(\Al,\bl)$ is generated by the $\left[(M,K);\left(\frac{E_i'}{E_i}\right)_{1\leq i\leq n}\right]$ defined by a $\Q$SK--pair $(M,K)\in\Ens(\Al,\bl)$ 
and elementary null LP--surgeries $\left(\frac{E_i'}{E_i}\right)$.
\end{corollary}
\begin{proof}
Consider $\left[(M,K);\left(\frac{A_i'}{A_i}\right)_{1\leq i\leq n}\right]\in\F_n(\Al,\bl)$. By Theorem \ref{thdecsur}, for each $i$, $A_i$ and $A_i'$ 
can be obtained from one another by a finite sequence of elementary surgeries or their inverses. 
Write $A_1'=A_1{\left(\frac{E_1'}{E_1}\right)}\dots{\left(\frac{E_k'}{E_k}\right)}$. 
For $0\leq j\leq k$, set $B_j=A_1{\left(\frac{E_1'}{E_1}\right)}\dots{\left(\frac{E_j'}{E_j}\right)}$. 
Then:
$$\left[(M,K);\left(\frac{A_i'}{A_i}\right)_{1\leq i\leq n}\right]=\sum_{j=1}^k \left[(M,K)\left(\frac{B_{j-1}}{B_0}\right);\frac{E_j'}{E_j},\left(\frac{A_i'}{A_i}\right)_{2\leq i\leq n}\right].$$
Decompose each surgery $\left(\frac{A_i'}{A_i}\right)$ in this way and conclude with:
$$\left[(M,K);\frac{E'}{E},\left(\frac{A_i'}{A_i}\right)_{2\leq i\leq n}\right]=-\left[(M,K)\left(\frac{E'}{E}\right);\frac{E}{E'},\left(\frac{A_i'}{A_i}\right)_{2\leq i\leq n}\right].$$
\end{proof}

Let $\F_0^{\QS}$ be the rational vector space generated by all $\Q$--spheres up to orientation-preserving homeomorphism. 
Let $(\F_n^{\QS})_{n\in\N}$ be the filtration of $\F_0^{\QS}$ defined by LP--surgeries, as before Definition \ref{definv}. 
Let $\displaystyle \G_n^{\QS}=\frac{\F_n^{\QS}}{\F_{n+1}^{\QS}}$ be the associated quotients. 
\begin{lemma}[\cite{M2} Proposition 1.8] \label{lemma07}
 For each prime integer $p$, let $B_p$ be a $\Q$--ball such that $H_1(B_p;\Z)\cong\frac{\Z}{p\Z}$. 
Then $\displaystyle \G_1^{\QS}=\bigoplus_{p\textrm{ prime}} \Q \left[S^3;\frac{B_p}{B^3}\right]$.
\end{lemma}

\begin{lemma} \label{lemmagenus0}
 For each prime integer $p$, let $B_p$ be a $\Q$--ball such that $H_1(B_p;\Z)\cong\frac{\Z}{p\Z}$. 
Let $(M,K)$ be a $\Q$SK--pair in $\Ens(\Al,\bl)$. Let $B$ be a $\Q$--ball. Let $\left(\frac{A_i'}{A_i}\right)_{1\leq i<n}$ be disjoint null LP--surgeries in $(M,K)$. 
Then $\displaystyle \left[(M,K);\frac{B}{B^3},\left(\frac{A_i'}{A_i}\right)_{1\leq i<n}\right]$ is a rational linear combination of the 
$\displaystyle \left[(M,K);\frac{B_p}{B^3},\left(\frac{A_i'}{A_i}\right)_{1\leq i<n}\right]$ and elements of $\F_{n+1}(\Al,\bl)$.
\end{lemma}
\begin{proof}
By Lemma \ref{lemma07}, there is a relation:
$$\left[S^3;\frac{B}{B^3}\right]=\sum_{p\textrm{ prime}} a_p \left[S^3;\frac{B_p}{B^3}\right] +\sum_{j\in J} b_j \left[N_j;\frac{C_j'}{C_j},\frac{D_j'}{D_j}\right],$$
where $J$ is a finite set, the $a_p$ and $b_j$ are rational numbers, the $a_p$ are all trivial except a finite number 
and, for $j\in J$, $\left[N_j;\frac{C_j'}{C_j},\frac{D_j'}{D_j}\right]\in\F_2^{\QS}$. 
For $I\subset\{1,..,n-1\}$, make the connected sum of each $\Q$--sphere in the relation with $M\left(\left(\frac{A_i'}{A_i}\right)_{i\in I}\right)$ 
to obtain:
$$\hspace{-10cm} \left[(M,K)\left(\left(\frac{A_i'}{A_i}\right)_{i\in I}\right);\frac{B}{B^3}\right]=$$
$$\hspace{1.8cm} \sum_{p\textrm{ prime}} a_p \left[(M,K)\left(\left(\frac{A_i'}{A_i}\right)_{i\in I}\right);\frac{B_p}{B^3}\right] 
+\sum_{j\in J} b_j \left[(M\sharp N_j,K)\left(\frac{A_i'}{A_i}\right)_{i\in I};\frac{C_j'}{C_j},\frac{D_j'}{D_j}\right].$$
Summing these equalities for all $I\subset\{1,..,n-1\}$, with appropriate signs, we get:
$$\hspace{-10cm} \left[(M,K);\frac{B}{B^3},\left(\frac{A_i'}{A_i}\right)_{1\leq i<n}\right]=$$
$$\hspace{1.8cm} \sum_{p\textrm{ prime}} a_p \left[(M,K);\frac{B_p}{B^3},\left(\frac{A_i'}{A_i}\right)_{1\leq i<n}\right] 
+\sum_{j\in J} b_j \left[(M\sharp N_j,K);\frac{C_j'}{C_j},\frac{D_j'}{D_j},\left(\frac{A_i'}{A_i}\right)_{1\leq i<n}\right].$$
\end{proof}
\begin{corollary} \label{cor03}
 Let $(M,K)\in\Ens(\Al,\bl)$. Let $\left(\frac{E_i'}{E_i}\right)_{1\leq i\leq n}$ be null elementary surgeries of genus 0 or 3. 
Then $\displaystyle \left[(M,K);\left(\frac{E_i'}{E_i}\right)_{1\leq i\leq n}\right]\in\varphi_n(\A_n^{aug}(\Al,\bl))$.
\end{corollary}
\begin{proof}
 Thanks to Lemma \ref{lemmagenus0}, it suffices to treat the case when the genus 0 surgeries are surgeries of type $\left(\frac{B}{B^3}\right)$ 
for a $\Q$--ball $B$ such that $|H_1(B;\Z)|$ is a prime integer. In this case, the considered bracket is the image 
of a diagram given as the disjoint union of 0--valent vertices and of $(\Al,\bl)$--colored diagrams of degree 1.
\end{proof}

To conclude the proof of Theorem \ref{thextend}, we need the next result about degree 1 invariants of {\em framed $\Q$--tori}, {\em i.e.} 
$\Q$--tori equipped with an oriented longitude. Note that any two framed $\Q$--tori have a canonical LP--identification of their boundaries, 
which identifies the fixed longitudes. LP--surgeries are well-defined on framed $\Q$--tori and we have an associated notion of finite type invariants. 
\begin{lemma}[\cite{M2} Corollary 5.10] \label{lemmainvtori}
For any prime integer $p$, let $M_p$ be a $\Q$--sphere such that $H_1(M_p;\Z)\cong\Z/p\Z$. 
Let $T_0$ be a framed standard torus. 
If $\mu$ is a degree 1 invariant of framed $\Q$--tori such that $\mu(T_0)=0$ and $\mu(T_0\sharp M_p)=0$ 
for any prime integer $p$, then $\mu=0$.
\end{lemma}

\begin{proof}[Proof of Theorem \ref{thextend}]
Take $\lambda\in(\F_n(\Al,\bl))^*$ such that $\lambda(\F_{n+1}(\Al,\bl))=0$. Assume $\lambda(\varphi_n(\A_n^{aug}(\Al,\bl)))=0$. 
In order to prove that $\varphi_n$ is onto, it is enough to prove that $\lambda=0$. 
Thanks to Corollary \ref{cordecel}, it suffices to prove that $\lambda$ vanishes on the brackets defined by elementary surgeries. 
For elementary surgeries of genus 0 and 3, this follows from Corollary \ref{cor03}. 

Consider a bracket $\left[(M,K);\left(\frac{T_{d_i}}{T_i}\right)_{1\leq i\leq k},\left(\frac{E_i'}{E_i}\right)_{1\leq i\leq n-k}\right]$, where $(M,K)\in\Ens(\Al,\bl)$, 
the $T_i$ are standard tori null in $M\setminus K$, the $T_{d_i}$ are $d_i$--tori for some positive integers $d_i$ 
and the $\left(\frac{E_i'}{E_i}\right)$ are null elementary surgeries of genus $0$ or $3$. 
By induction on $k$, we will prove that $\lambda$ vanishes on this bracket. We have treated the case $k=0$. 
Assume $k>0$. Fix a parallel of $T_1$. If $T$ is a framed $\Q$--torus, set:
$$\bar{\lambda}(T)=\lambda\left(\left[(M,K);\frac{T}{T_1},\left(\frac{T_{d_i}}{T_i}\right)_{2\leq i\leq k},\left(\frac{E_i'}{E_i}\right)_{1\leq i\leq n-k}\right]\right),$$
where the LP--identification $\partial T\cong\partial T_1$ identifies the prefered parallels. Then $\bar{\lambda}$ is a degree~$1$ 
invariant of framed $\Q$--tori:
$$\bar{\lambda}\left( \left[T;\frac{B_1}{A_1},\frac{B_2}{A_2}\right]\right) = 
-\lambda\left(\left[(M,K)\left(\frac{T}{T_1}\right);\frac{B_1}{A_1},\frac{B_2}{A_2},\left(\frac{T_{d_i}}{T_i}\right)_{2\leq i\leq k},\left(\frac{E_i'}{E_i}\right)_{1\leq i\leq n-k}\right]\right) =0.$$
Moreover, we have $\bar{\lambda}(T_1)=0$ and, by induction, $\bar{\lambda}\left(T_1\left(\frac{B_p}{B^3}\right)\right)=0$. 
Thus, by Lemma \ref{lemmainvtori}, $\bar{\lambda}=0$.
\end{proof}

\begin{proof}[Proof of Theorem \ref{thG1}]
Theorem \ref{thextend} provides a surjective map $\varphi_1\hspace{-3.7pt}: \A_1^{aug}(\Al,\bl)\twoheadrightarrow\G_1(\Al,\bl)$. Thanks to Lemma 
\ref{lemmaodd}, we have $\A_1^{aug}(\Al,\bl)=\oplus_{p\textrm{ prime}}\Q\, \bullet_p$. Hence $\G_1(\Al,\bl)$ is generated by the images 
of the diagrams $\bullet_p$, which are the brackets $\left[(M,K);\frac{B_p}{B^3}\right]$ for all prime integers $p$, with any $(M,K)\in\Ens(\Al,\bl)$. 

For any prime integer $p$, define a $\Q$--linear map $\nu_p:\F_0\to\Q$ by setting $\nu_p(M,K)=v_p(|H_1(M;\Z)|)$ for all $\Q$SK--pair $(M,K)$, 
where $v_p$ denotes the $p$--adic valuation. 
By \cite[Proposition 1.9]{M2}, the $\nu_p$ are degree 1 invariants of $\Q$--spheres, 
hence they are also degree 1 invariants of $\Q$SK--pairs. This implies that the family $\left(\left[(M,K);\frac{B_p}{B^3}\right]\right)_{p\textrm{ prime}}$ is free 
in $\G_1(\Al,\bl)$.
\end{proof}

    \section{Extension of the Lescop/Kricker invariant} \label{secZLes}

In this section, we prove Theorem \ref{thunivinv}.

Given two invariants $\lambda_1$ and $\lambda_2$ of $\Q$SK--pairs, define their product on any $\Q$SK--pair $(M,K)$ by $(\lambda_1\lambda_2)(M,K)=\lambda_1(M,K)\lambda_2(M,K)$ 
and extend to $\F_0$ by linearity. 
The following lemma is classical and holds for any objects and any invariants with values in some ring (see for instance \cite[Lemma 6.2]{M2}).
\begin{lemma} \label{lemmaprodinvariants}
 $$\left(\prod_{j=1}^n\lambda_j\right)\left(\left[\left(M,K\right);\left(\frac{B_i}{A_i}\right)_{i\in I}\right]\right)\hspace{10cm}$$
 $$=\sum_{\emptyset=J_0\subset\dots\subset J_n=I}
 \prod_{j=1}^n\lambda_j\left(\left[(M,K)\left(\left(\frac{B_i}{A_i}\right)_{i\in J_{j-1}}\right);\left(\frac{B_i}{A_i}\right)_{i\in J_j\setminus J_{j-1}}\right]\right)$$
\end{lemma}
This lemma implies in particular that a product of finite type invariants is a finite type invariant whose degree is at most the sum of the degrees of the factors.

\begin{proof}[Proof of Theorem \ref{thunivinv}]
 We begin with a preliminary remark about the invariant $Z$. It follows from the last point in Theorem \ref{thinvariantZ} that $Z_n\circ\varphi_n$ vanishes 
 on diagrams that contain isolated vertices. Now, the degree $n$ part of $Z^{aug}$ is given by:
 $$Z_n^{aug}=\sum_{k=0}^n\sum_{\substack{p_1<\dots<p_s \\ \textrm{ prime integers }}}
  \sum_{\substack{t_1+\dots+t_s=n-k\\t_i>0}}Z_k \sqcup\left(\coprod_{i=1}^s\frac{1}{t_i!}(\rho_{p_i})^{t_i}\right).$$
 That $Z_n^{aug}$ vanishes on $\F_{n+1}$ follows from Lemma \ref{lemmaprodinvariants}. 
 
 Let us compute $Z_n^{aug}\circ\varphi_n$. Let $D$ be an $(\Al,\bl)$--augmented diagram of degree~$n$. Write $D$ as the disjoint union of its Jacobi part $D_J$ 
 and its $0$--valent part $D_{\bullet}$. 
 Apply Lemma~\ref{lemmaprodinvariants}, noting that for a term in the right hand side of the obtained equality to be non trivial:
 \begin{itemize}
  \item each bracket must have exactly the order of the corresponding invariant, 
  \item each invariant $\rho_p$ must be evaluated on a bracket associated with the diagram $\bullet_p$,
  \item the invariant $Z_k$ must be evaluated on a bracket associated with a diagram without isolated vertices.
 \end{itemize}
 It follows that $Z_n^{aug}\circ\varphi_n(D)=(Z_k\circ\varphi_k(D_J))\sqcup D_{\bullet}=\psi_k(D_J)\sqcup D_{\bullet}=\psi_n(D)$.  
\end{proof}

	\section{Inverse of the map $\psib_n$} \label{secpsi}

In this section, we prove Theorem \ref{th3n}. To this end, we construct the inverse of the map $\psib_n$. The rough idea is to open the edges of a given 
$\delta$--colored diagram, inserting univalent vertices whose fixed equivariant linking is the label of the initial edge. We need some preliminaries. 

\begin{proposition} \label{prop2vertices}
 Fix a Blanchfield module $(\Al,\bl)$. Assume $\Al$ is a direct sum $\Al=\Al'\oplus\Al''$ orthogonal with respect to the Blanchfield form. 
 Let $D$ and $D'$ be $(\Al,\bl)$--colored diagrams which differ only by the labels of their univalent vertices, {\em i.e.} $D$ and $D'$ have the same underlying graph, 
 with a common set $V$ of univalent vertices, the same orientations and edges labels, and the same linkings between the univalent vertices. 
 Further assume that:
 \begin{itemize}
  \item there are two vertices $v$ and $w$ in $V$ whose labels in $D$ and $D'$ are elements of $\Al'$,
  \item for all other vertices in $V$, the labels in $D$ and $D'$ are equal and are elements of $\Al''$,
  \item for any $u\in V$ different from $v$ and $w$, we have $f_{uv}=0$ and $f_{uw}=0$ for $D$ and $D'$.
 \end{itemize}
 Then $D$ and $D'$ are equal in $\A_n(\Al,\bl)$, where $n$ is the degree of $D$ and $D'$.
\end{proposition}

\newcommand{\diag}[3]{\hspace{-1ex}\raisebox{-0.5cm}{
\begin{tikzpicture} [scale=0.8]
 \draw (0,0) node {\textnormal{\DJ{}}};
 \draw (-1,-0.05) node {$\scriptstyle{\bullet}$};
 \draw (1,-0.05) node {$\scriptstyle{\bullet}$};
 \draw (-1,0) -- (-0.3,0) (0.3,0) -- (1,0);
 \draw (-1,0) node[left] {$#1$};
 \draw (1,0) node[right] {$#2$};
 \draw[dashed] (-1,-0.2) .. controls +(0.3,-0.3) and +(-0.3,-0.3) .. (1,-0.2);
 \draw[white,line width=10pt] (0.3,-0.4) -- (0.7,-0.4);
 \draw (0.55,-0.4) node {$#3$};
\end{tikzpicture}}}
We first prove a few lemmas in the setting of the proposition.
In the following, we denote by $\diag{\gamma}{\eta}{f}$ the diagram identical to $D$ except for the labellings of $v$ and $w$, which are 
$\gamma$ and $\eta$ respectively, and the linking $f_{v,w}$, which is equal to $f$. 

We will use the structure of the Blanchfield module recalled in the next theorem. 
The {\em dual} of a polynomial $P(t)\in\Qt$ is the polynomial $\bar{P}(t)=P(t^{-1})$. The polynomial $P$ is {\em symmetric} if $\bar{P}(t)=at^kP(t)$ for some $a\in\Q$ and $k\in\Z$. 
\begin{theorem}[\cite{M1} Proposition 1.2 $\&$ Theorem 1.3] \label{thstructureBmod}
 The Blanchfield module $(\Al,\bl)$ of a $\Q$SK--pair is an orthogonal direct sum of:
 \begin{itemize}
  \item cyclic submodules $\frac{\Qt}{(\pi^n)}\gamma$ where $n$ is a positive integer, $\pi$ is either a symmetric prime polynomial with $\pi(\pm1)\neq0$, or $(t+2+t^{-1})$, 
   or a product of two dual non-symmetric prime polynomials, and $\bl(\gamma,\gamma)=\frac{P}{\pi^n}$ for some polynomial $P$ symmetric and prime to $\pi$,
  \item submodules $\frac{\Qt}{((t+1)^m)}\rho\oplus\frac{\Qt}{((t+1)^m)}\rho'$ where $m$ is an odd positive integer, $\bl(\rho,\rho)=0$, $\bl(\rho',\rho')=0$, 
   $\bl(\rho,\rho')=\frac{1}{(t+1)^m}$.
 \end{itemize}
\end{theorem}

\begin{lemma} \label{lemmadecvert}
 Assume $\Al'=\Al_1\oplus^\perp\Al_2$. If $\gamma\in\Al_1$ and $\eta \in \Al_2$, then: 
 $$\diag{\gamma}{\eta}{0}=0.$$
\end{lemma}
\begin{proof}
 Apply the Aut relation with the automorphism of $(\Al,\bl)$ given by the opposite of the identity on $\Al_1$ and the identity on $\Al_2\oplus\Al''$. 
\end{proof}

\begin{corollary} \label{cordecvert}
 Assume $\Al'$ is the orthogonal direct sum of submodules $\Al_i$, $i=1,\dots,k$. Let $\gamma, \eta \in \Al'$. Write $\gamma=\sum_{i=1}^k\gamma_i$ 
 and $\eta=\sum_{i=1}^k\eta_i$ with $\gamma_i,\eta_i\in\Al_i$. Then: 
 $$\diag{\gamma}{\eta}{f}=\sum_{i=1}^k\diag{\gamma_i}{\eta_i}{f_i},$$
 for all families of rational fractions $f_i$ such that $\bl(\gamma_i,\eta_i)=f_i\ mod\ \Qt$ and $\sum_{i=1}^kf_i=f$.
\end{corollary}

\begin{lemma} \label{lemmahol}
 If $\gamma, \eta \in \Al'$ and $P\in\Qt$, then:
 $$\diag{P\gamma}{\eta}{f}=\diag{\gamma}{\bar{P}\eta}{f}.$$
\end{lemma}
\begin{proof}
 In the case where $P$ is a power of $t$, apply the Aut relation with the automorphism of $(\Al,\bl)$ given by multiplication by some power of $t$ 
 on $\Al'$ and identity on $\Al''$. Conclude with the LV relation.
\end{proof}

\begin{corollary} \label{corhol}
 Assume $\displaystyle\Al'=\frac{\Qt}{(\pi)}\gamma$. Then $D=\diag{\gamma}{P\gamma}{f}$ for some $P\in\Qt$, with $f=f_{vw}^D$.
\end{corollary}

\begin{lemma} \label{lemmahol2}
 Assume $\displaystyle\Al'=\frac{\Qt}{((t+1)^m)}\rho\oplus\frac{\Qt}{((t+1)^m)}\rho'$ with $m$ odd, $\bl(\rho,\rho)=0$, $\bl(\rho',\rho')=0$, 
 $\bl(\rho,\rho')=\frac{1}{(t+1)^m}$. Then $D=\diag{\rho}{Q\rho'}{f}$ for some $Q\in\Qt$, with $f=f_{vw}^D$.
\end{lemma}
\begin{proof}
 Write $D=\diag{\nu}{\nu'}{f}$, with $\nu=A\rho+B\rho'$ and $\nu'=A'\rho+B'\rho'$. Applying the Aut relation with the automorphism given by $\rho\mapsto2\rho$, 
 $\rho'\mapsto\frac{1}{2}\rho'$, and identity on $\Al''$, we see that the diagrams $\diag{A\rho}{A'\rho}{0}$ and $\diag{B\rho'}{B'\rho'}{0}$ are trivial. 
 Hence we can decompose $D$ as follows:
 $$D=\diag{A\rho}{B'\rho'}{f_1}+\diag{B\rho'}{A'\rho}{f_2}.$$
 Now the automorphism given by $\rho\mapsto\rho'$, $\rho'\mapsto t^{m}\rho$, and identity on $\Al''$ gives $$\diag{B\rho'}{A'\rho}{f_2}=\diag{Bt^{m}\rho}{A'\rho'}{f_2}.$$ 
 Thanks to Lemma \ref{lemmahol}, we get $D=\diag{\rho}{P\rho'}{f},$ with $P=\bar{A}B'+\bar{B}t^{-m}A'$. 
\end{proof}

\begin{proof}[Proof of Proposition \ref{prop2vertices}]
 For $\pi\in\Qt$, the {\em $\pi$--component} of a $\Qt$--module is the submodule of its elements of order some power of $\pi$. 
 Any Blanchfield module is the direct sum of its $\pi$--components, where $\pi$ runs through all 
 prime symmetric polynomials (including $t+1$) and all products of two dual prime non-symmetric polynomials. 
 Thanks to Corollary \ref{cordecvert}, we can assume that $\Al'$ is reduced to one $\pi$--component. 
 \paragraph{First case: $\pi(-1)\neq0$.}
 The module $\Al'$ can be written as an orthogonal direct sum:
 $$\Al'=\bigoplus_{i=1}^p\frac{\Qt}{(\pi^{n_i})}\gamma_i$$
 with $\bl(\gamma_i,\gamma_i)=\frac{P_i}{\pi^{n_i}}$ for some symmetric polynomial $P_i$ prime to $\pi$, and $n:=n_1=\dots=n_q>n_{q+1}\geq\dots\geq n_p$. 
 Replacing $\gamma_1$ by some rational multiple if necessary, we can assume that $\sum_{i=1}^qP_i$ is prime to $\pi$. Set $\gamma=\sum_{i=1}^p\gamma_i$. 
 Then $\bl(\gamma,\gamma)=\frac{P}{\pi^{n}}$ with $P$ symmetric and prime to $\pi$. It follows that the submodule $\langle\gamma\rangle$ of $\Al'$ generated by 
 $\gamma$ has a trivial intersection with its orthogonal $\langle\gamma\rangle^\perp$, thus:
 $$\Al'=\langle\gamma\rangle\oplus^\perp\langle\gamma\rangle^\perp.$$
 
 By Corollaries \ref{cordecvert} and \ref{corhol}, we can decompose $D$ as $D=\sum_{i=1}^p\diag{\gamma_i}{Q_i\gamma_i}{f_i}$ for some polynomials $Q_i$. 
 Corollary \ref{cordecvert} gives $D=\diag{\gamma}{\eta}{f}$ with $\eta=\sum_{i=1}^pQ_i\gamma_i$ and $f=\sum_{i=1}^pf_i$. 
 Write $\eta=A\gamma+\mu$ with $\mu\in\langle\gamma\rangle^\perp$. Since $\diag{\gamma}{\mu}{0}=0$ by Lemma \ref{lemmadecvert}, we get $D=\diag{\gamma}{A\gamma}{f}$. 
 Similarly $D'=\diag{\gamma}{B\gamma}{f}$.  The condition on $f$ implies $\frac{AP}{\pi^n}=\frac{BP}{\pi^n}\ mod\ \Qt$, thus $A=B\ mod\ \pi^n$ and $A\gamma=B\gamma$.
 
 \paragraph{Second case: $\pi=t+1$.}
 In that case, the decomposition of $\Al'$ may involve non-cyclic submodules. We have $\Al'=\Al_1\oplus^\perp\Al_2$ with:
 $$\Al_1=\left({\bigoplus_{i=1}^p}^\perp\frac{\Qt}{(t+2+t^{-1})^{n_i}}\gamma_i\right) \qquad 
 \Al_2=\left({\bigoplus_{j=1}^k}^\perp\left(\frac{\Qt}{(t+1)^{m_j}}\rho_j\oplus\frac{\Qt}{(t+1)^{m_j}}\rho_j'\right)\right)$$
 with $\bl(\gamma_i,\gamma_i)=\frac{P_i}{(t+2+t^{-1})^{n_i}}$, $P_i(-1)\neq0$, $\bl(\rho_j,\rho_j)=0$, $\bl(\rho_j',\rho_j')=0$, 
 $\bl(\rho_j,\rho_j')=\frac{1}{(t+1)^{m_j}}$, $n_1=\dots=n_q>n_{q+1}\geq\dots\geq n_p$, $m_1\geq\dots\geq m_k$, $m_j$ odd. 
 We can assume $\sum_{i=1}^qP_i$ prime to $(t+1)$. Set $\gamma=\sum_{i=1}^p\gamma_i$ and $\rho=\sum_{j=1}^k\rho_j$. 
 
 Proceeding as in the first case, applications of Corollaries \ref{cordecvert} and \ref{corhol} and Lemma \ref{lemmahol2} give:
 $$D=\diag{\gamma}{\alpha}{f_1}+\diag{\rho}{\beta}{f_2},$$
 with $\alpha\in\Al_1$ and $\beta\in\Al_2$. Finally, $D=\diag{(\gamma+\rho)}{\eta}{f}$ with $\eta\in\Al'$ and $f=f_{vw}^D$. Similarly 
 $D'=\diag{(\gamma+\rho)}{\eta'}{f}$ with $\eta'\in\Al'$. 
 
 First assume $2n_1>m_1$. We have $\bl(\gamma+\rho,\gamma+\rho)=\sum_{i=1}^p\frac{P_i}{(t+2+t^{-1})^{n_i}}=\frac{P}{(t+2+t^{-1})^{n_1}}$ with $P(-1)\neq0$. 
 We get $\Al'=\langle\gamma+\rho\rangle\oplus^\perp\langle\gamma+\rho\rangle^\perp$ and we conclude as in the first case.
 
 Now assume $m_1>2n_1$. It is easily checked that $\langle\gamma+\rho,\rho_1'\rangle\cap\langle\gamma+\rho,\rho_1'\rangle^\perp=0$. 
 Hence $\Al'=\langle\gamma+\rho,\rho_1'\rangle\oplus^\perp\langle\gamma+\rho,\rho_1'\rangle^\perp$, and we can assume 
 $\eta,\eta'\in\langle\gamma+\rho,\rho_1'\rangle$. By Theorem \ref{thstructureBmod}, there is a basis $(\mu,\mu')$ of $\langle\gamma+\rho,\rho_1'\rangle$ 
 such that $\bl(\mu,\mu)=0$, $\bl(\mu',\mu')=0$, and $\bl(\mu,\mu')=\frac{1}{(t+1)^{m_1}}$. By Lemma~\ref{lemmahol2}, we have 
 $D=\diag{\mu}{A\mu'}{f}$ and $D'=\diag{\mu}{B\mu'}{f}$. Since the linking $f$ is the same, we get $A=B\ mod\ (t+1)^{m_1}$ and $A\mu'=B\mu'$.
\end{proof}

Let us fix some notations. Let $n$ be an even positive integer and $N\geq3n/2$. For $i=1,\dots,N$, let $(\Al_i,\bl_i)$ be a copy of $(\Al,\bl)$ 
and fix an isomorphism $\xi_i:(\Al,\bl)\fl{\cong}(\Al_i,\bl_i)$. Let $(\Abb)$ be the orthogonal direct sum of the $(\Al_i,\bl_i)$. Define permutation 
automorphisms $\xi_{ij}$ of $(\Abb)$ by $\xi_j\circ\xi_i^{-1}$ on $\Al_i$, $\xi_i\circ\xi_j^{-1}$ on $\Al_j$ and identity on the other $\Al_\ell$'s.
Given a diagram $D$ with set of univalent vertices $V$, denote by $D((\gamma_v)_{v\in V},(f_{vw})_{v\neq w\in V})$ the diagram obtained 
from $D$ by replacing the label of the vertex $v$ by $\gamma_v$ and the linking between $v$ and $w$ by $f_{vw}$. If all the linkings are the same as in $D$, 
we drop this part of the notation. 
\begin{definition} \label{defdistributed}
 An $(\Abb)$--colored diagram $D$ is {\em distributed} if there are a decomposition of the set of univalent vertices of $D$ as $V=\sqcup_{i=1}^{|V|/2}\{v_i,w_i\}$ 
 and indices $\ell_i$ with $\ell_i\neq\ell_j$ if $i\neq j$ such that the labels of $v_i$ and $w_i$ are elements of $\Al_{\ell_i}$ for all $i$ and the linkings between vertices 
 in different pairs are trivial.
\end{definition}

\begin{proposition} \label{propdist}
 The space $\A_n(\Abb)$ is generated by distributed $(\Abb)$--colored diagrams. 
\end{proposition}
\begin{proof}
 Let $D$ be an $(\Abb)$--colored diagram of degree $n$. First note that $D$ has $n$ trivalent vertices and each univalent vertex is related to a trivalent vertex by an edge 
 since we avoid struts, hence $D$ has at most $3n$ univalent vertices. We shall prove that $D$ is a linear combination of distributed diagrams. Thanks to the LV relation, 
 we can assume that all labels of univalent vertices of $D$ are elements of the $\Al_i$'s. Thanks to the LD and LV relations, we can assume that all univalent vertices 
 have non-trivial labels and the linking $f_{vw}$ is trivial if $v$ and $w$ are labelled in different $\Al_i$'s. If $D$ has an odd number of univalent vertices 
 labelled in some $\Al_i$, application of the automorphism given by opposite identity on $\Al_i$ and identity on the other $\Al_j$'s shows it is trivial. 
 Assume $D$ has an even number of univalent vertices labelled in each $\Al_i$. Let $i$ be an index such that the number of univalent vertices of $D$ 
 labelled in $\Al_i$ is maximal; denote $2s$ this number. If $s>1$, there is an $\Al_j$ that contains no labels of univalent vertices of $D$. 
 Consider the following automorphism $\chi_{ij}$ of $(\Abb)$:
 $$\chi_{ij}(\gamma):=\left\lbrace\begin{array}{l l}
                                       x\gamma+y\xi_j\circ\xi_i^{-1}(\gamma) & \textrm{ if } \gamma\in\Al_i \\
                                       y\xi_i\circ\xi_j^{-1}(\gamma)-x\gamma & \textrm{ if } \gamma\in\Al_j \\
                                       \gamma & \textrm{ if } \gamma\in\Al_\ell \textrm{ with } \ell\neq i,j 
                                      \end{array}\right.,$$
 where $x$ and $y$ are positive rational numbers such that $x^2+y^2=1$.
 Apply the Aut relation with $\chi_{ij}$ to $D$ and use the LV relation to express $D=D((\gamma_v)_{v\in V})$ as the sum of $x^{2s}D$, 
 $y^{2s}D((\xi_{ij}(\gamma_v))_{v\in V})$ and a linear combination $C$ of diagrams with strictly less than $2s$ vertices 
 in $\Al_i$ and in $\Al_j$. Now $D$ and $D((\xi_{ij}(\gamma_v))_{v\in V})$ are equal thanks to the Aut relation with $\xi_{ij}$. 
 It follows that $D$ is a rational multiple of $C$. Conclude by iterating. 
\end{proof}
\begin{remark}
 In the case of $\Z$SK--pairs, Proposition \ref{propdist} is the point in this section that does not work. Indeed, this proposition uses automorphisms $\chi_{ij}$ 
 whose definition is based on rational numbers $x$ and $y$ that are not integers. Thus it is not clear whether such isomorphisms are induced by isomorphisms 
 of the underlying integral Blanchfield module. For instance, consider the integral Blanchfield module $(\Al_\Z,\bl_\Z)$ defined by 
 $$\Al_\Z=\frac{\Zt}{(\delta)}\gamma\oplus^\perp\frac{\Zt}{(\delta)}\eta, \textrm{ with } \delta(t)=t-1+t^{-1} \textrm{ and }
   \bl_\Z(\gamma,\gamma)=\bl_\Z(\eta,\eta).$$
 Then any isomorphism of $(\Al_\Z,\bl_\Z)$ preserves the given direct sum decomposition. 
 Indeed, an isomorphism of $(\Al_\Z,\bl_\Z)$ has the following form:
 $$\left\lbrace \begin{array}{l} \gamma\mapsto P\gamma+Q\eta \\ \eta\mapsto R\gamma+S\eta \end{array} \right.,$$
 with $P\bar{P}+Q\bar{Q}=1$, $R\bar{R}+S\bar{S}=1$ and $P\bar{R}+Q\bar{S}=0$, where the polynomials are considered in $\Zt/(\delta)$. 
 Since $\delta$ has degree~$2$, one can write $P(t)=at+b$ and $Q(t)=ct+d$ with $a,b,c,d\in\Z$. This gives:
 $$P\bar{P}+Q\bar{Q}=a^2+b^2+ab+c^2+d^2+cd=\frac{1}{2}\big((a+b)^2+a^2+b^2+(c+d)^2+c^2+d^2\big).$$
 If $PQ\neq0$, then $a\neq0$ or $b\neq0$, and $c\neq0$ or $d\neq0$. It follows that $P\bar{P}+Q\bar{Q}\geq2$, contradicting the first condition on $P$ and $Q$. 
 Hence $PQ=0$ and the conditions on the polynomials $P,Q,R,S$ give $P=S=0$ or $Q=R=0$. 
\end{remark}

Recall the map $\iota_n:\A_n(\Al,\bl)\to\A_n(\Abb)$ is defined on diagrams by $\iota_n(D((\gamma_v)_{v\in V}))=D((\xi_1(\gamma_v))_{v\in V})$.
\begin{proposition}
 If $D$ is an $(\Al,\bl)$--colored diagram of degree $n$ with an even number of univalent vertices:
 $$\iota_n(D((\gamma_v)_{v\in V},(f_{vw})_{v\neq w\in V}))=
   \frac{1}{s!}\sum_{\sigma\in\Upsilon}D((\xi_{\sigma(v)}(\gamma_v))_{v\in V},(\delta_{\sigma(v)\sigma(w)}f_{vw})_{v\neq w\in V}),$$
 where $s=|V|/2$ and $\Upsilon=\left\lbrace\sigma:V\to\{1,\dots,s\}\textrm{ such that }|\sigma^{-1}(i)|=2\textrm{ for all }i=1,\dots,s\right\rbrace$. 
\end{proposition}
\begin{proof}
 We apply the method of the previous proposition with precise computations. We indeed prove a slightly more general result. 
 Consider an $(\Abb)$--colored diagram $D=D((\gamma_v)_{v\in V\sqcup W},(f_{vw})_{v\neq w\in V\sqcup W})$ with $|V|=2s$, $\gamma_v\in\Al_1$ if $v\in V$, 
 $\gamma_w\in\Al_i$ with $i>s$ if $w\in W$ and $f_{vw}=0$ if $v\in V$ and $w\in W$. We prove by induction on $s$ that in $\A_n(\Abb)$:
 $$D=\frac{1}{s!}\sum_{\sigma\in\Upsilon}D\big((\xi_{1\sigma(v)}(\gamma_v))_{v\in V}\cup(\gamma_w)_{w\in W},(\delta_{\sigma(v)\sigma(w)}f_{vw})_{v\neq w\in V}\cup(f_{vw})_{v\neq w\in W}\big),$$
 where the non-indicated linkings are trivial. We will use that our formulas remain valid when permuting the indices of the $\Al_i$'s, without mentioning it. 
 
 The result is trivial if $s=1$. Take $s>1$. 
 Applying the Aut relation with $\chi_{12}$ to $D$, we get:
 $$D=\sum_{k=0}^s\sum_{\substack{V=V_1\sqcup V_2 \\ |V_1|=2k}}x^{2k}y^{2(s-k)}D((\xi_{11}(\gamma_v))_{v\in V_1}\cup(\xi_{12}(\gamma_v))_{v\in V_2}\cup(\gamma_w)_{w\in W}),$$
 with, for the diagram in the right hand side, the linking $\delta_{\sigma(v)\sigma(w)}f_{vw}$ if $v\neq w$ are both in $V_1$ or both in $V_2$, $f_{vw}$ if $v\neq w$ are both in $W$ 
 and $0$ otherwise. Now apply twice the induction hypothesis with $V_1$ and $V_2$ instead of $V$ to obtain:
 $$(1-x^{2s}-y^{2s})D=\hspace{12cm}$$
 $$\sum_{k=1}^{s-1}\sum_{\substack{V=V_1\sqcup V_2 \\ |V_1|=2k}}\frac{x^{2k}y^{2(s-k)}}{k!(s-k)!}\sum_{\sigma\in\Upsilon_1}\sum_{\nu\in\Upsilon_2}
 D((\xi_{1\sigma(v)}(\gamma_v))_{v\in V_1}\cup(\xi_{1\nu(v)}(\gamma_v))_{v\in V_2}\cup(\gamma_w)_{w\in W}),$$
 with the required linkings, where $\Upsilon_1$ (resp. $\Upsilon_2$) is defined as $\Upsilon$ with $V_1$ and $\{1,\dots,k\}$ (resp. $V_2$ and $\{k+1,\dots,s\}$) 
 instead of $V$ and $\{1,\dots,s\}$. To conclude, note that each diagram in the right hand side occurs once for each value of $k$.
\end{proof}

\begin{proof}[Proof of Theorem \ref{th3n}]
 Define the inverse $\varPhi_n$ of the map $\psib_n:\A_{n}(\Abb)\to\A_n(\delta)$ in the following way. Given a $\delta$--colored diagram $D$ 
 of degree $n$, denote by $e_i$, $i=1,\dots,k$ its edges whose label is non-polynomial. ``Open'' each such edge $e_i$ as represented in Figure \ref{figopenedge}, 
 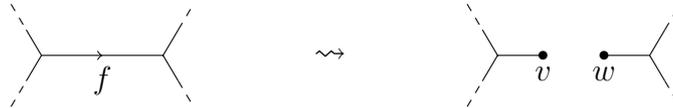
\begin{figure}[htb] 
\begin{center}
\begin{tikzpicture} [scale=0.2]
\begin{scope} [xshift=-21cm]
 \draw[dashed] (0,3.4) -- (1,1.7) (0,-3.4) -- (1,-1.7) (11,1.7) -- (12,3.4) (11,-1.7) -- (12,-3.4);
 \draw (1,1.7) -- (2,0) -- (10,0) -- (11,1.7) (1,-1.7) -- (2,0) (10,0) -- (11,-1.7);
 \draw[->] (5.9,0) -- (6,0);
 \draw (6,0) node[below] {$f$};
\end{scope}
\draw (0,0) node {$\rightsquigarrow$};
\begin{scope} [xshift=9cm]
 \draw [dashed] (0,3.4) -- (1,1.7) (0,-3.4) -- (1,-1.7) (13,1.7) -- (14,3.4) (13,-1.7) -- (14,-3.4);
 \draw (1,1.7) -- (2,0) -- (5,0) (9,0) -- (12,0) -- (13,1.7) (1,-1.7) -- (2,0) (12,0) -- (13,-1.7);
 \draw (5,0) node {$\scriptstyle{\bullet}$};
 \draw (9,0) node {$\scriptstyle{\bullet}$};
 \draw (5,0) node[below] {$v$};
 \draw (9,0) node[below] {$w$};
\end{scope}
\end{tikzpicture}
\end{center} \caption{Opening an edge} \label{figopenedge}
 \end{figure}
 label the created vertices $v$ and $w$ by some $\gamma_v$ and $\gamma_w$ in $\Al_i$ such that $\bl(\gamma_v,\gamma_w)=f\ (mod\ \Qt)$, 
 and fix the linking $f_{vw}=f$. Such $\gamma_v$ and $\gamma_w$ always exist: note that $\gamma_v$ can be chosen of order $\delta$, the annihilator of $\Al_i$, 
 use the non-degeneracy of the Blanchfield form and the fact that the denominator of $f$ has to divide $\delta$. Fix the other linkings to $0$, 
 so that we obtain a distributed diagram $\varPhi_n(D)$. It does not depend on the numbering of the edges of $D$ thanks to the Aut relation 
 in $\A_n(\Abb)$ with the permutation automorphisms $\xi_{ij}$. It is also independent of the choice of labels $\gamma_v,\gamma_w\in\Al_i$ 
 by Proposition \ref{prop2vertices}. Note that these independence arguments imply that any distributed diagram in $\A_{n}(\Abb)$ is a $\varPhi_n(D)$. 
 
 We have to check that the relations defining $\A_n(\delta)$ are respected. It is immediate for AS and IHX. OR follows from the rule $f_{wv}(t)=f_{vw}(t^{-1})$ 
 on linkings. Hol and Hol' are recovered {\em via} Hol and EV. LE follows from LE when the involved edges have polynomial labels, from LD when one 
 of the involved edges has a polynomial label and from LV when the involved edges have non-polynomial labels. In this latter case, note that one can open 
 this edge with the same label on one univalent vertex for the three diagrams. 
 Finally, we have a well-defined map $\varPhi_n:\A_n(\delta)\to\A_n(\Abb)$ satisfying by construction $\psib_n\circ\varPhi_n=\textrm{id}_{\A_n(\delta)}$. 
 Now $\varPhi_n$ is surjective by Proposition \ref{propdist}. Thus $\psib_n$ and $\varPhi_n$ are inverse isomorphisms.  
\end{proof}

We end with a few results that derive from the above argument and prove useful in the study of the structure of the diagram space $\A_n(\Abb)$, as it appears in \cite{AM}. 
The first one gives a simplified presentation of $\A_n(\Abb)$. 
\begin{proposition} \label{proppresalt}
 Keeping the notations fixed before Definition \ref{defdistributed}, we have:
 $$\A_n(\Abb)\cong\frac{\Q\langle\textrm{ degree }n\textrm{ distributed }(\Abb)\textrm{--colored diagrams }\rangle}
 {\Q\langle\textrm{AS, IHX, Hol, OR, LE, LV, EV, LD, Aut}_{res}\rangle}$$
 where the relation Aut$_{res}$ is the Aut relation restricted to the following automorphisms of $(\Abb)$: permutation automorphisms $\xi_{ij}$, 
 automorphisms fixing one $\Al_i$ setwise and the others pointwise.
 Moreover, if $(\Al,\bl)$ is cyclic, we can further restrict the Aut relation to: permutation automorphisms $\xi_{ij}$, 
 multiplication by $t$ or $-1$ on one $\Al_i$ and identity on the others.
\end{proposition}
\begin{proof}
 We can see that the space defined by the given presentation is isomorphic to $\A_n(\delta)$ as we did for $\A_n(\Abb)$ in the proof of Theorem \ref{th3n}. 
 At the level of generators, the proof of Theorem \ref{th3n} only uses distributed diagrams and at the level of relations, one has to check that the proof 
 of Proposition \ref{prop2vertices} only uses the Aut relation with the allowed automorphisms.
\end{proof}

In order to study $\A_n(\Abb)$, it is natural and helpful to consider the fitration induced by the number of univalent vertices. 
For $k=0,\dots,3n$, let $\A_n^{(k)}(\Al,\bl)$ be the subspace of $\A_n(\Al,\bl)$ generated by diagrams with at most $k$ univalent vertices and set:
$$\widehat{\A}_n^{(k)}(\Al,\bl)=\frac{\Q \langle(\Al,\bl)\mbox{--colored diagrams of degree $n$ with at most $k$ univalent vertices}\rangle}
                                     {\Q \langle\mbox{AS, IHX, LE, OR, Hol, LV, EV, LD, Aut}\rangle}.$$
Similarly, let $\A_n^{(k)}(\delta)$ be the subspace of $\A_n(\delta)$ generated by diagrams with at most $k/2$ edges with a non-polynomial label and set: 
$$\widehat{\A}_n^{(k)}(\delta)=\frac{\Q \left\langle\mbox{\begin{tabular}{c} $\delta$--colored diagrams
    of degree $n$ \\ with at most $k/2$ edges with a non-polynomial label \end{tabular}}\right\rangle}{\Q \langle\mbox{AS, IHX, LE, OR, Hol, Hol'}\rangle}.$$
Recall that all these diagram spaces are trivial when $n$ is odd. Moreover, the number of trivalent vertices and the number of univalent vertices in a uni-trivalent 
graph have the same parity. So we are only interested in cases where $n$ and $k$ are even. 
Define a map $\widehat{\psi}_n^{(k)}:\widehat{\A}_n^{(k)}(\Al,\bl)\to\widehat{\A}_n^{(k)}(\delta)$ via pairings of vertices as $\psi_n$ was defined 
in Section \ref{subsecLesKri}. 
\begin{proposition}
 Let $n$, $k$ and $K$ be integers such that $0\leq k\leq 3n$ and $k\leq 2K$. Then:
 \begin{itemize}
  \item the isomorphism $\psib_n:\A_{n}(\Abb)\to\A_n(\delta)$ induces an isomorphism $\A_n^{(k)}(\Abb)\cong\A_n^{(k)}(\delta)$,
  \item the map $\widehat{\psi}_n^{(k)}:\widehat{\A}_n^{(k)}\left((\Al,\bl)^{\oplus K}\right)\to\widehat{\A}_n^{(k)}(\delta)$ is an isomorphism,
  \item the space $\widehat{\A}_n^{(k)}\left((\Al,\bl)^{\oplus K}\right)$ admits the following presentation:
    $$\widehat{\A}_n^{(k)}\left((\Al,\bl)^{\oplus K}\right)\cong
    \frac{\Q\left\langle\begin{array}{c}
                         \textrm{degree }n\textrm{ distributed }\left((\Al,\bl)^{\oplus K}\right)\textrm{-colored diagrams} \\
                         \textrm{ with at most $k$ univalent vertices }
                        \end{array}\right\rangle} 
    {\Q\langle\textrm{AS, IHX, Hol, OR, LE, LV, EV, LD, Aut}_{res}\rangle}.$$
 \end{itemize}
\end{proposition}
\begin{proof}
 The first point is due to the fact that $\psib_n$ identifies the $(\Al,\bl)$--colored diagrams with at most $k$ univalent vertices with the $\delta$--colored diagrams 
 with at most $k/2$ edges with a non-polynomial label. The last two points follow from the same argument as Theorem~\ref{th3n} and Proposition~\ref{proppresalt}.
\end{proof}
In \cite{AM}, a Blanchfield module $(\Al,\bl)$ is explicited for which $\widehat{\A}_2^{(4)}(\Abb)\ncong\A_2^{(4)}(\Abb)$ and this provides the example with a negative 
answer to Question \ref{qupsin} mentioned in the introduction.

\def\cprime{$'$}
\providecommand{\bysame}{\leavevmode ---\ }
\providecommand{\og}{``}
\providecommand{\fg}{''}
\providecommand{\smfandname}{\&}
\providecommand{\smfedsname}{\'eds.}
\providecommand{\smfedname}{\'ed.}
\providecommand{\smfmastersthesisname}{M\'emoire}
\providecommand{\smfphdthesisname}{Th\`ese}


\begin{thebibliography}{Mou12b}

\bibitem[AL05]{AL}
{\scshape E.~Auclair {\normalfont \smfandname} C.~Lescop} -- {\og Clover
  calculus for homology 3--spheres via basic algebraic topology\fg},
  \emph{Algebraic \& Geometric Topology} \textbf{5} (2005), p.~71--106.

\bibitem[AM17]{AM}
{\scshape B.~Audoux {\normalfont \smfandname} D.~Moussard} -- {\og Toward universality in degree 2 of the {K}ricker lift of the {K}ontsevich integral and the {L}escop equivariant invariant\fg}, to appear in \emph{International Journal of Mathematics}, arXiv:1710.09730, 2017.

\bibitem[Auc06]{Auc}
{\scshape E.~Auclair} -- {\og {Surfaces et invariants de type fini en
  dimension~3}\fg}, {Ph.D. Thesis}, Universit\'e Joseph Fourier, Grenoble,
  2006.

\bibitem[Bla57]{Bla}
{\scshape R.~C. Blanchfield} -- {\og Intersection theory of manifolds with
  operators with applications to knot theory\fg}, \emph{Annals of Mathematics
  (2)} \textbf{65} (1957), p.~340--356.
  
\bibitem[CM00]{CM}
{\scshape T.~Cochran {\normalfont \smfandname} P.~Melvin} -- {\og Finite type 
  invariants of 3--manifolds\fg}, \emph{Inventiones mathematicae} \textbf{140} 
  (2000), no.~1, p.~45--100.

\bibitem[GGP01]{GGP}
{\scshape S.~Garoufalidis, M.~Goussarov {\normalfont \smfandname} M.~Polyak} --
  {\og Calculus of clovers and finite type invariants of 3--manifolds\fg},
  \emph{Geometry \& Topology} \textbf{5} (2001), p.~75--108.

\bibitem[GK04]{GK}
{\scshape S.~Garoufalidis {\normalfont \smfandname} A.~Kricker} -- {\og A
  rational noncommutative invariant of boundary links\fg}, \emph{Geometry \&
  Topology} \textbf{8} (2004), p.~115--204.

\bibitem[GR04]{GR}
{\scshape S.~Garoufalidis {\normalfont \smfandname} L.~Rozansky} -- {\og The
  loop expansion of the {K}ontsevich integral, the null-move and
  {$S$}--equivalence\fg}, \emph{Topology} \textbf{43} (2004), no.~5,
  p.~1183--1210.

\bibitem[Hab00a]{Habe}
{\scshape N.~Habegger} -- {\og {M}ilnor, {J}ohnson, and tree level perturbative
  invariants\fg}, preprint, 2000.

\bibitem[Hab00b]{Hab}
{\scshape K.~Habiro} -- {\og Claspers and finite type invariants of links\fg},
  \emph{Geometry and Topology} \textbf{4} (2000), p.~1--83.

\bibitem[{Kri}00]{Kri}
{\scshape A.~{Kricker}} -- {\og {The lines of the {K}ontsevich integral and
  {R}ozansky's rationality conjecture}\fg}, arXiv:math/0005284, 2000.

\bibitem[Les11]{Les2}
{\scshape C.~Lescop} -- {\og Invariants of knots and 3--manifolds derived from
  the equivariant linking pairing\fg}, in \emph{Chern--{S}imons gauge theory: 20
  years after}, AMS/IP Stud. Adv. Math., vol.~50, AMS, Providence, RI, 2011,
  p.~217--242.

\bibitem[Les13]{Les3}
\bysame , {\og A universal equivariant finite type knot invariant defined from
  configuration space integrals\fg}, arXiv:1306.1705, 2013.

\bibitem[Mat87]{Mat}
{\scshape S.~V. Matveev} -- {\og Generalized surgery of three-dimensional
  manifolds and representations of homology spheres\fg}, \emph{Mathematical notes of the Academy of Sciences of the USSR
  } \textbf{42} (1987), no.~2, p.~651--656.

\bibitem[Mou12a]{M2}
{\scshape D.~Moussard} -- {\og Finite type invariants of rational homology
  3--spheres\fg}, \emph{Algebraic \& Geometric Topology} \textbf{12} (2012),
  no.~4, p.~2389--2428.

\bibitem[Mou12b]{M1}
\bysame , {\og On {A}lexander modules and {B}lanchfield forms of
  null-homologous knots in rational homology spheres\fg}, \emph{Journal of Knot
  Theory and its Ramifications} \textbf{21} (2012), no.~5, p.~1250042, 21.

\bibitem[Mou15]{M3}
\bysame , {\og {Rational Blanchfield forms, S--equivalence, and null
  LP--surgeries}\fg}, \emph{{Bulletin de la Soci\'et\'e Math\'ematique de
  France}} \textbf{143} (2015), no.~2, p.~403--431.

\bibitem[Mou17]{M6}
\bysame , {\og Splitting formulas for the rational lift of the {K}ontsevich
  integral\fg}, to appear in \emph{Algebraic \& Geometric Topology}, arXiv:1705.01315, 2017.

\bibitem[Oht96]{Oht4}
{\scshape T.~Ohtsuki} -- {\og Finite type invariants of integral homology
  {$3$}--spheres\fg}, \emph{Journal of Knot Theory and its Ramifications} \textbf{5} (1996),
  no.~1, p.~101--115.

\end{thebibliography}
\end{document}